\documentclass[11pt,reqno]{amsart}

\newtheorem{thm}{Theorem}[section]
\newtheorem*{thm*}{Theorem}
\newtheorem{lem}[thm]{Lemma}
\newtheorem*{lem*}{Lemma}

\newtheorem{cor}[thm]{Corollary}

\newtheorem{prop}[thm]{Proposition}

\theoremstyle{definition}

\newtheorem{case}{Case}\renewcommand{\thecase}{}
\newtheorem*{case*}{Case}

\newtheorem{defn}[thm]{Definition}
\newtheorem*{defn*}{Definition}

\newtheorem*{exmp*}{Example}

\newtheorem{rmk}[thm]{Remark}
\newtheorem*{rmk*}{Remark}
\newtheorem{step}{Step}\renewcommand{\thestep}{}

\theoremstyle{remark}

\makeatletter
\def\alphenumi{
  \def\theenumi{\alph{enumi}}
  \def\p@enumi{\theenumi}
  \def\labelenumi{(\@alph\c@enumi)}}
\makeatother

\makeatletter
\def\thecase{\@arabic\c@case}
\makeatother

\makeatletter
\def\thestep{\@arabic\c@step}
\makeatother

\newtheorem{rem}[thm]{Remark}

\newcommand{\R}{\mathbb R}

\newcommand{\C}{\mathscr C}
\newcommand{\E}{\mathscr E}

\newcommand{\HH}{\mathbb H}
\newcommand{\F}{{\mathscr F}(u)}

\newcommand{\Bo}{B_\rho(P_0)}
\newcommand{\Bop}{B_\rho^+(P_0)}
\newcommand{\Boy}{B_{\rho y_0}(P_0)}
\newcommand{\Boo}{B_\rho}

\newcommand{\Bdy}{B_{dy_0}(P_0)}
\newcommand{\Qdy}{B_{dy_0}(P_0)}

\newcommand{\lo}{l_{P_0}}

\newcommand{\dist}{\mathrm{dist}}

\newcommand\loc{\operatorname{loc}}

\newcommand\sA{{\mathscr{A}}}

\newcommand\sC{{\mathscr{C}}}

\newcommand\sE{{\mathscr{E}}}
\newcommand\sF{{\mathscr{F}}}

\newcommand\sO{{\mathscr{O}}}

\newcommand\fw{{\mathfrak{w}}}

\numberwithin{equation}{section}

\renewcommand\emptyset{\varnothing}

\usepackage{amssymb}
\usepackage{epstopdf}
\usepackage{graphicx}
\usepackage{hyperref}
\usepackage{mathrsfs}
\usepackage[usenames]{color}
\usepackage{url}
\usepackage{verbatim}

\hypersetup{pdftitle={$C^{1,1}$ regularity for degenerate elliptic obstacle problems}}
\hypersetup{pdfauthor={Panagiota Daskalopoulos and Paul M. N. Feehan}}

\begin{document}

\title[$C^{1,1}$ regularity for degenerate elliptic obstacle problems]
{$C^{1,1}$ regularity for degenerate elliptic obstacle problems}

\author[P. Daskalopoulos]{Panagiota Daskalopoulos}
\address{Department of Mathematics, Columbia University, New York, NY 10027, United States of America}
\email{pdaskalo@math.columbia.edu}

\author[Paul M. N. Feehan]{Paul M. N. Feehan}
\address{Department of Mathematics, Rutgers, The State University of New Jersey, 110 Frelinghuysen Road, Piscataway, NJ 08854-8019, United States of America}
\email{feehan@math.rutgers.edu}

\curraddr{School of Mathematics, Institute for Advanced Study, Princeton, NJ 08540}
\email{feehan@math.ias.edu}

\subjclass[2010]{Primary 35J70, 35J86, 49J40, 35R45; Secondary 35R35, 49J20, 60J60}

\keywords{Degenerate elliptic differential operator; Free boundary problem; Heston stochastic volatility process; Obstacle problem; Variational inequality; Weighted H\"older and Sobolev spaces}

\thanks{Daskalopoulos was partially supported by NSF grant DMS-0905749. Feehan was partially supported by NSF grant DMS-1059206 and the Max Planck Institut f\"ur Mathematik in der Naturwissenschaft.}

\date{Journal of Differential Equations (2016), \url{http://dx.doi.org/10.1016/j.jde.2015.11.037}. This version: December 20, 2015, incorporating final galley proof corrections.}

\begin{abstract}
The Heston stochastic volatility process is a degenerate diffusion process where the degeneracy in the diffusion coefficient is proportional to the square root of the distance to the boundary of the half-plane. The generator of this process with killing, called the elliptic Heston operator, is a second-order, degenerate-elliptic partial differential operator, where the degeneracy in the operator symbol is proportional to the distance to the boundary of the half-plane. In mathematical finance, solutions to the obstacle problem for the elliptic Heston operator correspond to value functions for perpetual American-style options on the underlying asset. With the aid of weighted Sobolev spaces and weighted H\"older spaces, we establish  the optimal $C^{1,1}$  regularity (up to the boundary of the half-plane) for solutions to obstacle problems for the elliptic Heston operator when the obstacle functions are sufficiently smooth.
\end{abstract}

\maketitle
\tableofcontents

\section{Introduction}
\label{sec-intro}
In \cite{Daskalopoulos_Feehan_statvarineqheston} (see also \cite{Feehan_maximumprinciple, Feehan_Pop_higherregularityweaksoln}), the authors established the existence and  uniqueness of a solution in a weighted Sobolev space\footnote{See section \ref{sec:Spaces} for their definitions.}, $u \in H_{\loc}^2(\sO\cup\Gamma_0,\fw)\cap H^1(\sO,\fw)$,  to the obstacle problem,
\begin{equation}
\label{eqn-obst1}
\begin{aligned}
\min\{Au-f,u-\psi\} &= 0 \quad \hbox{a.e. on }\sO,
\\
u &= g \quad\hbox{on } \Gamma_1,
\end{aligned}
\end{equation}
for the \emph{Heston operator} \cite{Heston1993},
\begin{equation}
\label{eqn-Heston}
Au := -\frac{y}{2}\left(u_{xx} + 2\varrho \sigma u_{xy} + \sigma^2 u_{yy}\right) - \left(r-q-\frac{y}{2}\right)u_x - \kappa(\theta-y)u_y + ru,
\end{equation}
on a subdomain  $\sO$ (possibly unbounded) of the  upper half-plane $\HH := \R \times (0,\infty)$, where $f:\sO\to\R$ is a source function, $g:\Gamma_1\to\R$ prescribes a Dirichlet boundary condition along $\Gamma_1 := \HH\cap\partial\sO$,  and $\psi:\sO\cup\Gamma_1\to\R$ is an obstacle function which is compatible with $g$ in the sense that $\psi\leq g$ on $\Gamma_1$. The differential operator $A$ given in  \eqref{eqn-Heston}  is elliptic  on $\sO$ but becomes degenerate along $\bar\Gamma_0$, where $\Gamma_0$ denotes the interior of $\{y=0\}\cap\partial\sO$. Because
$\kappa \theta >0$ (see assumption \eqref{eqn-coeff} below), {\em no boundary condition} is prescribed along the portion $\bar\Gamma_0$ of the boundary $\partial\sO = \bar\Gamma_0\cup\Gamma_1$ of $\sO$.

The operator  $A$ is  the generator of the two-dimensional Heston stochastic volatility process with killing, a degenerate diffusion process well known in mathematical finance and a paradigm for a broad class of degenerate diffusion processes. The coefficients defining $A$ in \eqref{eqn-Heston} are constants assumed throughout this article to obey
\begin{equation}
\label{eqn-coeff}
\sigma \neq 0, \quad -1< \varrho < 1, \quad r \geq 0, \quad q \geq 0, \quad \kappa>0, \quad  \theta>0,
\end{equation}
while their financial meaning is described in \cite{Heston1993}. For a detailed introduction to the Heston operator and the obstacle  problem \eqref{eqn-obst1}, we refer the reader to our article  \cite{Daskalopoulos_Feehan_statvarineqheston}.

In this article,  we will establish $C^{1,1}_s$ regularity on $\sO\cup\Gamma_0$ and a priori $C^{1,1}_s$ estimates for the solution $u$ to \eqref{eqn-obst1} on subdomains $U\Subset \sO\cup\Gamma_0$. We use $C^{1,1}_s$ to indicate a weighted H\"older norm and corresponding H\"older space which are distinct from the usual $C^{1,1}$ H\"older norm and H\"older space and which take into account the degeneracy of the operator, $A$, along $y=0$ --- see section \ref{sec:Spaces} for their definition. In the case of a uniformly elliptic operator on a bounded domain, interior $C^{1,1}$ regularity was established by Brezis and Kinderlehrer \cite{Brezis_Kinderlehrer_1973} (see also \cite[Theorem 1.4.1]{Friedman_1982} for a statement of their result and an exposition of their proof), while global $C^{1,1}$ regularity, given a Dirichlet boundary condition, was established by Jensen \cite{Jensen_1980} (see also \cite[Theorem 4.38]{Troianiello} for a statement of his result and an exposition of his proof), recalling that \cite[p. 23]{Friedman_1982}, for a bounded domain $U\subset\R^n$, one has $W^{2,\infty}(U) = C^{1,1}(\bar U)$.
To the best of the authors' knowledge, however, our article is the first to establish $C^{1,1}$ regularity of a solution to an obstacle problem defined by a \emph{degenerate} elliptic operator with a
boundary degeneracy of the kind in \eqref{eqn-Heston},
despite the importance of this question in applications to American-style option pricing problems for asset prices modeled by stochastic volatility processes \cite{Heston1993}.
While Danielli, Garofalo, and Salsa \cite{Danielli_Garofalo_Salsa_2003} also obtain optimal regularity for a solution to a degenerate obstacle problem, their operator is hypoelliptic and so there is no overlap between the techniques in \cite{Danielli_Garofalo_Salsa_2003} and those of the present article.

For interior $C^{1,1}$ regularity, the case of a uniformly elliptic operator on a bounded domain reduces, by standard methods (see, for example, \cite{Blanchet_2006b, Troianiello}), to the case of the Laplace operator and ingenious techniques introduced by Caffarelli \cite{Caffarelli_jfa_1998} greatly simplify the proof of interior $C^{1,1}$ regularity for solutions to an obstacle problem in this case. We shall adapt Caffarelli's approach in our article but, because of the degeneracy of our operator, $A$, along $y=0$, careful consideration must be given to  the different scaling of the equation near $y=0$. This scaling is reflected in the use of the {\em cycloidal} distance function, $s(\cdot,\cdot)$, defined in section \ref{sec:Spaces} and of weighted Sobolev, H\"older and $C^{1,1}$ spaces. Weighted Sobolev and H\"older spaces have been introduced previously (see, for example, \cite{Daskalopoulos_Feehan_statvarineqheston, DaskalHamilton1998, Feehan_Pop_regularityweaksoln, Koch}) in order to obtain sharp
estimates for solutions to equations  involving degenerate elliptic operators of the form \eqref{eqn-Heston} and their parabolic analogues.

Let $B_\rho(Q_0) := \{P\in \R^2: \dist(P,Q_0) < \rho\}$ denote the open ball with center $Q_0=(p_0,q_0)\in\R^2$ and radius $\rho>0$, and set
\begin{equation}
\label{eqn-halfball}
B_\rho^+(Q_0):=B_\rho(Q_0) \cap \HH.
\end{equation}
For a given radius $R_0>0$ and for any $R$ obeying $0<R<R_0$, we denote
\begin{equation}
\label{eq:domains}
V := B_{R_0}^+(Q_0) \quad\hbox{and}\quad U := B_R^+(Q_0).
\end{equation}
Throughout our article, we shall assume that $Q_0=(p_0,q_0) \in \bar\HH$ with $0\leq q_0\leq \Lambda$, for a positive constant\footnote{Note that $A$ in \eqref{eqn-Heston} is uniformly elliptic on $B_{R_0}(Q_0)$ when $q_0>R_0$, and results concerning regularity of solutions to \eqref{eqn-obst2} are then standard \cite{Friedman_1982, Troianiello} and so, for the purpose of this article, we could choose $\Lambda=R_0$ without loss of generality.} $\Lambda$. We shall abuse notation slightly and let $\Gamma_0$ denote the interiors of $\partial\HH\cap\partial\sO$, $\partial\HH\cap\partial V$, or $\partial\HH\cap\partial U$ when we write $\sO\cup\Gamma_0$, $V\cup\Gamma_0$, or $U\cup\Gamma_0$, respectively. The definitions of the weighted H\"older spaces, $C^\alpha_s(\bar V)$, $C^{2+\alpha}_s(\bar V)$ and $C^{1,1}_s(\bar V)$, which we require for the statement of the main result of this article below are collected in section \ref{sec:Spaces}.

\begin{thm}[Optimal regularity]
\label{thm-c11}
Let $R_0 > 0$ and $\Lambda>0$ and suppose $Q_0=(p_0,q_0) \in \bar\HH$ with $0 \leq q_0  \leq \Lambda$. Let $V$ be as in \eqref{eq:domains}. Assume that $u \in H^2(V,\fw) \cap C(\bar V)$ is a solution to the obstacle problem,
\begin{equation}
\label{eqn-obst2}
\min\{Au-f,u-\psi\} = 0 \quad \hbox{a.e. on } V,
\end{equation}
with $\psi \in C^{2+\alpha}_s(\bar V)$ and
$f \in C^{\alpha}_s(\bar V)$, for some
$\alpha \in (0,1)$. If $r > 0$ in \eqref{eqn-Heston}, then $u\in C^{1,1}_s(V\cup\Gamma_0)$
and there is a constant $C$, depending on $\alpha, R_0, \Lambda$, and the coefficients of the operator $A$, such that if $U$ is as in \eqref{eq:domains} with $R=R_0/2$, then
\begin{equation}
\label{eqn-c1100}
\| u \|_{C^{1,1}_s(\bar U)} \leq C\left(\| u \|_{C(\bar V)} + \| f \|_{C^\alpha_s(\bar V)} + \| \psi \|_{C^{1,1}(\bar V)}\right).
\end{equation}
\end{thm}

See Remark \ref{rem:Positive_coefficient_r} for comments regarding the hypotheses that $r>0$ and $\psi \in C_s^{2+\alpha}(\bar V)$. Theorem \ref{thm-c11} immediately yields

\begin{cor}[Optimal regularity]
\label{cor-c11}
Let $\sO \subset \HH$ be a bounded domain. Assume that $u \in H_{\mathrm{loc}}^2(\sO\cup\Gamma_0,\fw)\cap C(\sO\cup\Gamma_0)$ is a solution to the obstacle problem \eqref{eqn-obst2} on $\sO$ with
$\psi \in C_s^{2+\alpha}(\bar\sO)$ and $f \in C^{\alpha}_s(\bar\sO)$, for some $\alpha \in (0,1)$. Then,
$u\in C^{1,1}_s(\sO\cup\Gamma_0)$ and, for each precompact subdomain $\sO'\Subset\sO\cup\Gamma_0$, there is a constant $C$, depending on $\alpha, \sO', \sO$, and the coefficients of the operator $A$, such that
\begin{equation}
\label{eqn-c11domain}
\| u \|_{C^{1,1}_s(\bar\sO')} \leq C\left(\| u \|_{C(\bar\sO)} + \| f \|_{C^\alpha_s(\bar\sO)} + \| \psi \|_{C^{1,1}(\bar\sO)}\right).
\end{equation}
\end{cor}

\begin{rmk}[Hypotheses on regularity of the solution]
\label{rmk:thm-c11_continuity}
As we note following Definition \ref{defn:Weighted_Sobolev_spaces}, functions that belong to the weighted Sobolev space, $H^2(V,\fw)$, are continuous up to the boundary portion, $\Gamma_1$, but need not be continuous up to the boundary portion, $\bar\Gamma_0$, and this is the reason for our hypothesis in Theorem \ref{thm-c11} that the function $u$ belong to both $H^2(V,\fw)$ and $C(\bar V)$ and similarly in Corollary \ref{cor-c11}. The H\"older regularity results \cite[Theorem 1.20 and Corollary 1.21]{Feehan_Pop_regularityweaksoln} suggest that these continuity hypotheses may be relaxed with the aid of interior versions of those results.
\end{rmk}

\begin{rmk}[Regularity up to the corners where $\Gamma_0$ and $\Gamma_1$ meet]
\label{rmk:thm-c11_corners}
It is interesting to note that, even without an obstacle, it is difficult in general to establish higher-order regularity up to the corner points for a solution $u$ to the boundary-degenerate elliptic Heston equation and indeed this is not asserted by in Theorem \ref{thm-c11} or Corollary \ref{cor-c11}. In the article \cite{Feehan_Pop_regularityweaksoln}, Pop and the second author applied a version of Moser iteration for weak solutions to the boundary-degenerate elliptic Heston equation and obstacle problems to prove that a weak solution $u$ is $C_s^\alpha$-H\"older continuous, for \emph{some} $\alpha \in (0,1)$, and continuous in the usual sense up to the corners where $\Gamma_0$ and $\Gamma_1$ meet. This issue concerning regularity at the corners where $\Gamma_0$ and $\Gamma_1$ meet is discussed further in two other articles with Pop \cite{Feehan_Pop_elliptichestonschauder, Feehan_Pop_higherregularityweaksoln}.
\end{rmk}

\begin{rmk}[Regularity of the obstacle function]
One can speculate as to whether it might be possible to improve the estimates in Theorem \ref{thm-c11} and Corollary \ref{cor-c11} by replacing $\|\psi\|_{C^{1,1}(\bar V)}$ with $\|\psi\|_{C_s^{1,1}(\bar V)}$ on the right-hand side. It is not clear that simple refinements of our proofs would yield such an improvement, but that does not preclude the possibility that a more sophisticated proof could succeed.
\end{rmk}

Our proof of Theorem \ref{thm-c11} proceeds by adapting ideas of Caffarelli in \cite{Caffarelli_jfa_1998}; see also an exposition by Petrosyan in \cite{Petrosyan_2011_msrilectures}. However, because  our operator is degenerate, careful consideration must be given to the difference of the scaling of the equation in regions  close ($y$ small) and  away ($y \geq \rho >0$) from the portion of the boundary, $\{y=0\}\cap\partial V$, where $A$ becomes degenerate.

\subsection{Generalizations}
When the main result of our article (Theorem \ref{thm-c11}) is combined with Jensen's global $C^{1,1}$ regularity theorem \cite{Jensen_1980}, we see that $H^2(\sO,\fw)$ solutions, $u$, to \eqref{eqn-obst1} actually belong to
$C^{1,1}_s(\sO\cup\Gamma_0)\cap C^{1,1}(\sO\cup\Gamma_1)$ under hypotheses on $f$ and $\psi$ analogous to those stated in Theorem \ref{thm-c11}. By making further use of methods in \cite{Feehan_Pop_elliptichestonschauder}, it should follow that $u\in C^{1,1}_{s,\loc}(\bar\sO)$. Moreover, there is good reason to believe that  results on the regularity  of the free boundary for the obstacle problem defined by a non-degenerate elliptic or parabolic operator extend to degenerate operators of the kind considered in this article; see \cite{Petrosyan_Shagholian_Uraltseva} and references therein for the non-degenerate elliptic case and \cite{Laurence_Salsa_2009, Nystrom_2007} and references therein for the non-degenerate parabolic case. We shall leave consideration of these extensions to our future articles.

The solution, $u$, to \eqref{eqn-obst1} can be interpreted as the value function for a perpetual American-style option with payoff function, $\psi$ \cite{KaratzasShreve1998}.
The $C^{2+\alpha}_s(\bar V)$ regularity property assumed for the obstacle function, $\psi$, in the statement of Theorem \ref{thm-c11} does not reflect the more typical Lipschitz regularity for $\psi$ encountered in applications to mathematical finance, such as $\psi(x,y) = \max\{E-e^x,0\}$, where $E$ is a positive constant, in the case of a put option. Nevertheless, simple examples in this context \cite[\S 8.3]{Shreve2} and results of \cite{Laurence_Salsa_2009, Nystrom_2007} suggest that the solution, $u$, should nevertheless have the optimal $C^{1,1}_s$ regularity even when $\psi = \max\{E-e^x,0\}$. Again, we shall leave consideration of this question to our future articles.

We have chosen, in this article, to work with our model, the Heston operator $A$, because of its relevance to mathematical finance and reliance on results in our previous work \cite{Daskalopoulos_Feehan_statvarineqheston} and that of Feehan and Pop \cite{Feehan_Pop_higherregularityweaksoln, Feehan_Pop_regularityweaksoln, Feehan_Pop_elliptichestonschauder, PopThesis}. However, we expect that the $C^{1,1}_s$ regularity result and a priori estimate in Theorem \ref{thm-c11} may be easily  generalized to  higher dimensions and degenerate elliptic operators on  $\R^{n-1} \times (0,\infty)$ with variable coefficients,
$$
Au =  -x_n a_{ij} u_{x_ix_j} - b_iu_{x_i} + cu,
$$
under the assumptions that $(a_{ij})$ is strictly elliptic,  $b_n \geq \nu >0$,  for some constant $\nu >0$, and $c \geq 0$ and all coefficients are H\"older continuous of class $C^{\alpha}_s(\bar V)$, for some $\alpha \in (0,1)$. See \cite{Feehan_Pop_mimickingdegen_pde, Feehan_Pop_mimickingdegen_probability} for an analysis with applications to probability theory based on parabolic operators of this type.

\subsection{Outline of the article}
\label{subsec:Guide}
For the convenience of the reader, we provide a brief outline of the article. We begin in \S \ref{sec:Spaces} by reviewing our definitions of weighted H\"older spaces \cite{DaskalHamilton1998} and weighted Sobolev spaces \cite{Daskalopoulos_Feehan_statvarineqheston} which we shall need for this article. In \S \ref{sec:Regularity}, we review results from \cite{Feehan_maximumprinciple, Feehan_Pop_regularityweaksoln, Feehan_Pop_higherregularityweaksoln, Feehan_Pop_elliptichestonschauder} concerning existence, uniqueness, and regularity of solutions to the elliptic Heston equation on bounded subdomains of the upper half-plane; see also \cite{DaskalHamilton1998}. In \S \ref{sec:SupBounds}, we develop the key pointwise growth estimates (see Propositions \ref{prop-sup1} and \ref{prop-sup2}) for solutions to the obstacle problem for the elliptic Heston operator. We conclude in \S \ref{sec:ProofMainTheorem} with the proof of our main result, Theorem \ref{thm-c11}.

\subsection{Notation}
Throughout the rest of the article  we will set $Lu: =-Au$, where $A$ is given by \eqref{eqn-Heston} and we work with $L$ instead to facilitate comparisons with the methods of Caffarelli \cite{Caffarelli_jfa_1998} and the sign conventions therein. The operator $L$ is then given by
\begin{equation}
\label{eqn-oper}
Lu = \frac{y}2\left( u_{xx} + 2 \varrho  \sigma u_{xy} + \sigma^2 u_{yy} \right) +
\left(r  - q-   \frac y2\right)u_x + \kappa\left(\theta - y\right)u_y - ru,
\end{equation}
with coefficients which satisfy the assumption \eqref{eqn-coeff}.

We let $C=C(*,\ldots,*)$ denote a constant which depends at most on the quantities appearing on the parentheses. In a given context, constants denoted by $C, C', \cdots$ and so on may have different values depending on the same set of arguments and may increase from one inequality to the next. Constants with values denoted by $K, K',\cdots$ and so on are reserved for quantities which remain fixed. We let $C(L)$ denote a constant which may depend on one or more of the constant coefficients of the operator $L$ (that is, $q,r,\kappa,\theta,\varrho,\sigma$).

\subsection{Acknowledgments}
We are grateful to Arshak Petroysan for sharing Mathematica code from his lecture notes \cite{Petrosyan_2011_msrilectures} and which we adapted to create the figures in this article. We are also grateful to Camelia Pop for many helpful conversations. We are grateful to the anonymous referees for their careful reading of our manuscript and their comments.

\section{Weighted Sobolev and H\"older spaces}
\label{sec:Spaces}
In \cite{Daskalopoulos_Feehan_statvarineqheston} the authors defined the following weighted Sobolev spaces of functions on a possibly unbounded domain $\sO \subset \HH$.

\begin{defn}[Weighted Sobolev spaces]
\label{defn:Weighted_Sobolev_spaces}
Let $L^2(\sO,\fw)$ denote the Hilbert space of Borel measurable functions, $u : \sO \to \R$, such that
$$
\|u\|_{L^2(\sO,\fw)}:= \left(\int_\sO  u^2 \, \fw\,dx\,dy \right)^{1/2} < \infty,
$$
with weight function $\fw(x,y) := y^{\beta-1}e^{-\gamma|x|-\mu y}$, for $(x,y)\in\HH$, where $\beta := 2\kappa\theta/\sigma^2$ and $\mu := 2\kappa/\sigma^2$ and the constant $\gamma >0$ depends only on the coefficients of $A$. We define the vector spaces,
\begin{align*}
H^1(\sO,\fw)
&:= \left\{u \in L^2(\sO,\fw):\  y^{1/2}|Du|,\ (1+y)^{1/2} u \in L^2(\sO,\fw)\right\},
\\
H^2(\sO,\fw)
&:= \left\{u \in L^2(\sO,\fw):\  y|D^2u|,\ (1+y) |Du|,\ (1+y)^{1/2} u \in L^2(\sO,\fw)\right\},
\end{align*}
where $Du=(u_x,u_y)$ and $D^2u = (u_{xx}, u_{xy}, u_{yx}, u_{yy})$ are defined in the sense of distributions.
\end{defn}

When equipped with the norm,
$$
\|u\|_{H^2(\sO,\fw)} := \left(\int_\sO  \left(y^2|D^2u|^2 + (1+y)^2|Du|^2 + (1+y)u^2\right)\, \fw\,dx\,dy \right)^{1/2},
$$
one finds that $H^2(\sO,\fw)$ is a Hilbert space and, noting that $\sO$ has dimension two, $H^2(\sO,\fw) \subset C(\sO\cup\Gamma_1)$ via the embedding theorem for standard, unweighted Sobolev spaces \cite{Daskalopoulos_Feehan_statvarineqheston}, but elementary examples show that functions in $H^2(\sO,\fw)$ need not be continuous up to $\bar\Gamma_0$. We say that $u \in H_{\loc}^2(\sO \cup \Gamma_0)$, if $u \in H^2(U)$ for any subdomain $U \Subset \sO \cup \Gamma_0$.

We next define weighted $C^{1,1}$ and H\"older norms on a bounded domain $\sO \subset \HH$.

\begin{defn} [$C^{1,1}_s$ norm and Banach space]
\label{defn-c11}  We say that $u \in C^{1,1}_s(\bar \sO)$ if $u$ belongs to  $C^{1,1}(\sO) \cap C^1(\bar \sO)$ and
$$
\| u \|_{C^{1,1}_s(\bar \sO)} := \| yD^2 u \|_{L^\infty(\sO)} +  \| D u \|_{C(\bar\sO)} +  \|  u \|_{C(\bar\sO)}  < \infty.
$$
Also, we say  that $u \in C^{1,1}_s(\sO \cup \Gamma_0)$, if $u \in C^{1,1}_s(\bar U)$ for any subdomain $U \Subset \sO \cup \Gamma_0$.
\end{defn}

We recall the definition of the distance function, $s(\cdot,\cdot)$ on $\HH$, equivalent to the distance function defined by the \emph{cycloidal metric}, $y^{-1}(dx^2+dy^2)$ on $\HH$, and introduced by Daskalopoulos and Hamilton in \cite{DaskalHamilton1998} and by H.  Koch in \cite{Koch},
\begin{equation}
\label{eqn-dist}
s(z,z_0) := \frac{|x-x_0| + |y-y_0|}{\sqrt{y} + \sqrt{y_0} + \sqrt{|x-x_0| + |y-y_0|}}, \quad \forall z=(x,y), z_0=(x_0,y_0)\in \HH.
\end{equation}
This is the natural metric for our degenerate equation; see \cite{DaskalHamilton1998} for a discussion. The following weighted H\"older spaces were introduced by Daskalopoulos and Hamilton in \cite{DaskalHamilton1998}.

\begin{defn}[$C^\alpha_s$ and $C^{2+\alpha}_s$ norms and Banach spaces]
\label{defn:DHspaces}
Given $\alpha \in (0,1)$, we say that $u\in C^\alpha_s(\bar\sO)$ if $u \in C(\bar\sO)$ and
$$
\|u\|_{C^\alpha_s(\bar \sO)} := \| u \|_{C(\bar\sO)} + \sup_{\stackrel{z_1,z_2\in\sO}{z_1\neq z_2}}\frac{|u(z_1)-u(z_2)|}{s(z_1,z_2)^\alpha} < \infty.
$$
We say that $u \in C^{2+\alpha}_s(\bar\sO)$ if $u$ has continuous first and second derivatives, $Du, D^2u$, in $\sO$, and $Du, yD^2u$ extend continuously up to the boundary, $\partial\sO$, and the extensions belong to $C^\alpha_s(\bar\sO)$. We denote
$$
\| u \|_{C^{2+\alpha}_s(\bar\sO)} :=  \| u \|_{C^{\alpha}_s(\bar\sO)} + \| Du \|_{C^{\alpha}_s(\bar\sO)} + \| yD^2 u \|_{C^{\alpha}_s(\bar\sO)}.
$$
We say that\footnote{In \cite[p. 901]{DaskalHamilton1998}, when defining the spaces $C^\alpha_s(\sA)$ and $C^{2+\alpha}_s(\sA)$, it is assumed that $\sA$ is a compact subset of the \emph{closed} half-plane, $\{y\geq 0\}$.} $u \in C^\alpha_s(\sO \cup \Gamma_0)$ if $u \in C^\alpha_s(\bar U)$ for every subdomain $U \Subset \sO \cup \Gamma_0$ and similarly that
$u \in C^{2+\alpha}_s(\sO \cup \Gamma_0)$  if $u \in C^{2+\alpha}_s(\bar U)$ for every subdomain $U \Subset \sO \cup \Gamma_0$.
\end{defn}

One can show that $C^{1,1}_s(\bar\sO)$, $C^\alpha_s(\bar\sO)$, and $C^{2+\alpha}_s(\bar\sO)$ are Banach spaces when equipped with the indicated norms.

For any subset $S \subset \HH$, we let $C_b(S)$ denote the vector space of bounded, continuous functions on $S$.

\begin{rem} On any bounded subdomain $U \subset \HH$ we have,
\begin{equation}\label{eqn-cnorms}
c\,  |z-z_0| \leq  s( z,z_0) \leq \sqrt {|z-z_0|},
\end{equation}
for some positive constant $c:=c(\mbox{diam} (U))$ depending only on the diameter of  $U$. Hence, $C^\alpha(\bar U) \subset C^\alpha_s(\bar U) \subset C^{\alpha/2} (\bar U)$.
\end{rem}

\section {Schauder existence, uniqueness, and regularity results}
\label{sec:Regularity}
We collect some known results for solutions to the degenerate elliptic equation,
\begin{equation}
\label{eqn-dee}
Lv =f \quad \mbox{on }  V,
\end{equation}
where $V$ is as in \eqref{eq:domains}. These results will be used in the proof of Theorem \ref{thm-c11}. Theorems \ref{thm-dp-C0BoundaryData-apriori} and \ref{thm-dp-C0BoundaryData-existence} are proved in \cite{Feehan_Pop_elliptichestonschauder} and may be viewed as analogues of \cite[Theorems 6.2, 6.6, 6.13 and 6.14]{GilbargTrudinger} and a generalization of \cite[Theorem I.1.1]{DaskalHamilton1998}.

\begin{thm}[A priori Schauder interior estimate]
\label{thm-dp-C0BoundaryData-apriori}
(See \cite[Theorem 1.1]{Feehan_Pop_elliptichestonschauder}.)
Given  $f \in C^{\alpha}_s(V\cup\Gamma_0)$, where $V$ is as in \eqref{eq:domains}, and
a solution $u \in C^{2+\alpha}_s(V\cup\Gamma_0)$ to
$$
Lv=f \quad\hbox{on } V,
$$
there is a constant, $C$, depending at most on $\alpha,R,R_0,\Lambda$, and the coefficients of $L$, such that if $U$ is as in \eqref{eq:domains}, then
\begin{equation}
\label{eqn-dpC0-schauder}
\|u\|_{C^{2+\alpha}_s(\bar U) } \leq C\left(\|u\|_{C(\bar V)} + \|f\|_{C^\alpha_s(\bar V)}\right).
\end{equation}
\end{thm}

\begin{thm}[Existence of a solution to a Dirichlet problem with continuous boundary data]
\label{thm-dp-C0BoundaryData-existence}
(See \cite[Corollary 1.13]{Feehan_Pop_elliptichestonschauder}, \cite[Theorem 1.1]{Feehan_classical_perron_elliptic}.)
Given  $f \in C^{\alpha}_s(V\cup\Gamma_0)\cap C_b(V)$, where $V$ is as in \eqref{eq:domains}, and $g \in C_s^{2+\alpha}(V\cup \Gamma_0)\cap C_b(V\cup\Gamma_1)$,
there exists a unique solution $u\in  C^{2+\alpha}_s(V\cup\Gamma_0)\cap C_b(V\cup\Gamma_1)$ to the Dirichlet problem
\begin{equation}
\label{eq:dp}
Lv=f \quad\hbox{on } V \quad\hbox{and}\quad v=g \quad\hbox{on }\HH\cap\partial V.
\end{equation}
\end{thm}

The preceding results easily imply the following consequence when combined with a regularity theorem from \cite{Feehan_Pop_higherregularityweaksoln} and a maximum principle estimate from \cite{Feehan_maximumprinciple}.

\begin{prop}[Regularity and interior Schauder estimate]
\label{prop-c2a}
Let $f \in C^\alpha_s(V\cup\Gamma_0)$ and let $v \in H^2(V, \fw)$ be a solution to
$$
Lv=f \quad\hbox{a.e. on } V.
$$
Then, $v \in C^{2+\alpha}_s(V\cup\Gamma_0)$.
Moreover, there is a constant, $C$, depending at most on $\alpha,R,R_0,\Lambda$, and the coefficients of $L$, such that if $U$ is as in \eqref{eq:domains}, then
\begin{equation}
\label{eqn-schauder}
\| v \|_{C^{2+\alpha}_s(\bar U)} \leq C\left (\| v \|_{C(\bar V)} + \| f \|_{C^{\alpha}_s(\bar V)} \right).
\end{equation}
\end{prop}

\begin{proof}
Choose $R_1$ obeying $R\leq R_1 <R_0$ and let $V_1 := B^+_{R_1}(Q_0)$, so that $U\subseteqq V_1\Subset V\cup\Gamma_0$. Then $f \in C^\alpha_s(\bar V_1)$ and we may choose $w \in C^{2+\alpha}_s(V_1\cup\Gamma_0)\cap C_b(V_1\cup\Gamma_1)$ to be the unique solution to $Lw=f$ on $V_1$ and $w = 0$ on $\HH\cap\partial V_1$ provided by Theorem \ref{thm-dp-C0BoundaryData-existence}. Moreover, $v_0 := v-w \in H^2(V_1, \fw)$ is a solution to $Lv_0 = 0$ a.e. on $V_1$ and so, by \cite[Corollary 1.8]{Feehan_Pop_higherregularityweaksoln},
we have $v_0 \in C^\infty(V_1\cup\Gamma_0)$ and
thus $v = v_0 + w \in C^{2+\alpha}_s(V_1\cup\Gamma_0)$. Since $R_1$ is arbitrary,
we obtain $v \in C^{2+\alpha}_s(V\cup\Gamma_0)$, as desired. The a priori estimate \eqref{eqn-schauder} thus follows from the a priori estimate \eqref{eqn-dpC0-schauder} provided by Theorem \ref{thm-dp-C0BoundaryData-apriori}.
\end{proof}

\begin{rem}[Alternative proofs of regularity in  Proposition \ref{prop-c2a}]
We can avoid relying on the regularity result \cite[Corollary 1.8]{Feehan_Pop_higherregularityweaksoln} if we are given $v \in H^2(V, \fw)\cap C_b(V)$. Indeed, Theorem \ref{thm-dp-C0BoundaryData-existence} provides
a unique solution $\tilde v \in C^{2+\alpha}_s(V_1\cup\Gamma_0)\cap C_b(V_1\cup\Gamma_1)$ to $L\tilde v=f$ on $V_1$ and $\tilde v = v$ on $\HH\cap\partial V_1$. But then $\tilde v \in H^2(V_1,\fw)$ and by the weak maximum principle for $L$ acting on functions in $H^2(V_1,\fw)$ \cite[Lemma 6.13 \& Theorem 8.8]{Feehan_maximumprinciple}, we must have $v=\tilde v$ a.e. on $V_1$ and
thus $v \in C^{2+\alpha}_s(V_1\cup\Gamma_0)$.
\end{rem}

The following weak and strong  maximum  principles are shown in \cite{Feehan_maximumprinciple}.
Recall that if $v  \in C^{2+\alpha}_s(\sO\cup\Gamma_0)$, then $Dv \in C(\sO\cup\Gamma_0)$ and $yD^2v\in C(\sO\cup\Gamma_0)$ (by definition) while $yD^2v = 0$ on $\Gamma_0$ (see \cite[Proposition I.12.1]{DaskalHamilton1998} or \cite[Lemma 3.1]{Feehan_Pop_mimickingdegen_pde}).

\begin{thm}[Weak maximum principle for the Heston operator]
\label{thm-wmp}
\cite[Theorem 5.1]{Feehan_maximumprinciple}
Let $v  \in C^{2+\alpha}_s(\sO\cup\Gamma_0)\cap C_b(\sO\cup\Gamma_1)$ be a subsolution, $Lv \geq 0$ on $\sO$ and $v \leq 0$ on $\HH\cap \partial \sO$, for a bounded domain $\sO \subset \HH$.  Then $v \leq 0$ on $\sO$.
\end{thm}

\begin{thm}[Strong maximum principle for the Heston operator]
\label{thm-smp}
\cite[Theorem 4.10]{Feehan_maximumprinciple}
Let $v \in C^{2+\alpha}_s(\sO\cup\Gamma_0)$ be a subsolution, $Lv \geq 0$ on $\sO$, for a bounded, connected domain $\sO \subset \HH$. If $v$ achieves its maximum value at a point $P \in \sO \cup \Gamma_0$ and, in addition,  $v(P) \geq 0$ if $r >0$ (where $r$ is the coefficient of $L$ in \eqref{eqn-oper}),
then $v$ must be a constant on $\bar\sO$.
\end{thm}

We finish this section by showing  how to reduce to the case $f = 0$ in Theorem \ref{thm-c11}.

\begin{prop}[Reduction to a homogeneous obstacle problem]
\label{prop-fzero}
We may assume, without loss of generality, that $f = 0$ on $V$ in Theorem \ref{thm-c11}.
\end{prop}

\begin{proof}
Let $v \in C^{2+\alpha}_s(V\cup\Gamma_0)\cap C_b(V\cup\Gamma_1)$ be the solution to the Dirichlet problem $Lv = f$ on $V$ and $v=0$ on $\HH\cap\partial V$ (its existence follows from Theorem \ref{thm-dp-C0BoundaryData-existence}). It follows from \eqref{eqn-dpC0-schauder} that
$$
\| v \|_{C^{2+\alpha}_s(\bar U) } \leq  C\left(\| v \|_{C(\bar V)} + \| f \|_{C^\alpha(\bar V)}\right),
$$
and hence by the weak maximum principle estimate,
$$
\| v \|_{C(\bar V)} \leq \frac{1}{r}\| f \|_{C(\bar V)},
$$
provided by \cite[Proposition 2.2 (6) and Theorem 5.1]{Feehan_maximumprinciple}, noting that $r > 0$ by hypothesis in Theorem \ref{thm-c11}, we obtain
\begin{equation}
\label{eqn-c2a6}
\| v \|_{C^{2+\alpha}_s(\bar U) } \leq  C\| f \|_{C^\alpha(\bar V)},
\end{equation}
where in \eqref{eqn-c2a6} we use $C$ to denote a constant which depends at most on $\alpha$, $R_0$, $\Lambda$ and the coefficients  of $L$.

Now if $u$ is a solution to the obstacle problem \eqref{eqn-obst1} on $V$ as in Theorem \ref{thm-c11}, then $\bar u:=u-v$ is a solution to the obstacle problem \eqref{eqn-obst1} on $V$ with source function $\bar f=0$ on $V$ and obstacle $\bar \psi := \psi - v$ on $V$. If Theorem \ref{thm-c11} is proved for $\bar f = 0$ in place of $f$ on $V$, then $\bar u \in C^{1,1}_s(\bar U)$ and the estimate \eqref{eqn-c1100} for $\bar u$ yields
$$
\|\bar u\|_{C^{1,1}_s(\bar U)} \leq C\left(\|\bar u \|_{C(\bar V)} + \|\bar \psi \|_{C^{1,1}(\bar V)}\right).
$$
But $u = \bar u + v \in C^{1,1}_s(\bar U)$ and we obtain the estimate \eqref{eqn-c1100} for $u$ from the preceding inequality and the estimate \eqref{eqn-c2a6} for $v$, together with the weak maximum principle estimate for $v$.
\end{proof}

\begin{rem}[Role of the hypotheses that the coefficient $r$ is positive in Theorem \ref{thm-c11}]
\label{rem:Positive_coefficient_r}
We appeal to positivity of the coefficient $r$ in the statement of Theorem \ref{thm-smp} and the proof of Proposition \ref{prop-fzero}; we use the fact that $\psi \in C_s^{2+\alpha}(\bar V)$ when we appeal to Theorem \ref{thm-dp-C0BoundaryData-existence} to solve the Dirichlet problem \eqref{eqn-xi} with source term $L\psi \in C_s^\alpha(V\cup\Gamma_0)\cap C_b(V\cup\Gamma_1)$.
\end{rem}

Because of the reduction in Proposition \ref{prop-fzero}, we may assume without loss of generality that $u \in H^2(V,\fw) \cap C(\bar V)$ is  a solution to the  obstacle problem \eqref{eqn-obst1} with obstacle function $\psi \in C_s^{2+\alpha}(\bar V)$
and $f = 0$ on $V$, that is,
\begin{equation}
\label{eqn-obst_homog}
\min\{-Lu, u-\psi\} = 0 \quad \mbox{a.e. on } V.
\end{equation}
We make this assumption for the remainder of this article.

\section{Supremum bounds}
\label{sec:SupBounds}
We will assume, throughout this section, that $u$ is a solution to the obstacle problem \eqref{eqn-obst_homog} on $V$, where $V$ is as in \eqref{eq:domains}, and that all the assumptions of Theorem \ref{thm-c11} hold. Adopting the terminology of mathematical finance, we call
\begin{equation}
\label{eqn-om}
\C(u) = \{P \in V\cup\Gamma_0  :  u(P) > \psi(P)\}
\end{equation}
the {\em continuation region}  (or non-coincidence set),
\begin{equation}
\label{eqn-con}
\E(u) = \{P \in V\cup\Gamma_0  : u(P) = \psi(P)\}
\end{equation}
the {\em exercise region} (or coincidence set), and
\begin{equation}
\label{eqn-fb}
\F = (V\cup\Gamma_0)\cap \partial \C(u)
\end{equation}
the \emph{free boundary} (or optimal exercise boundary, as it is known in mathematical finance). From \eqref{eqn-obst_homog} and \eqref{eqn-om}, we see that
\begin{equation}
\label{eqn-obst3}
Lu \leq 0  \quad \mbox{a.e. on } V \quad \mbox{and} \quad Lu = 0 \quad \mbox{on } \C(u).
\end{equation}
Since $Lu=0$ on $\sC(u)$, it follows from Proposition \ref{prop-c2a} that $u$ is of class $C^{2+\alpha}_s$  on  $\C(u)$. (Actually one may also easily see  that $u$ is of class $ C^\infty$ on $\C(u)$.)

We will establish sharp growth estimates from above on $u-\psi$ near free boundary  points $P_0 \in \F$. Because of the degeneracy of our operator $L$,  we will need to scale our estimates in different ways, depending on the distance of $P_0$ from the boundary portion, $\bar\sO\cap\partial\HH = \{y=0\}$. Similar estimates in the non-degenerate case, where $L$ is the Laplace operator, $\Delta$, were established by Caffarelli in \cite{Caffarelli_jfa_1998}.

The first such estimate, in Proposition \ref{prop-sup1}, concerns with free boundary points $P_0=(x_0,y_0) \in \F$ with $y_0 >0$. To simplify the notation we will assume that $0< y_0 <1$. The estimate near any free boundary point $P_0=(x_0,y_0) \in \F$ with $y_0 >1$ can be shown similarly. We have the following analogue of \cite[Lemma 2]{Caffarelli_jfa_1998}; see also \cite[Lemma 1.6]{Petrosyan_2011_msrilectures} (where $L=\Delta$ and $\psi=0$).

\begin{prop}[Quadratic growth of solution near free boundary and away from degenerate boundary]
\label{prop-sup1}
Let $u$ be as in Theorem \ref{thm-c11} and let $P_0=(x_0,y_0) \in \F \cap V$ with $0 < y_0 <1$. Then there are constants $0 < \rho_0 <1$ and $0<C < \infty$, depending at most on the coefficients of $L$, such that if $B_{\rho_0 y_0}(P_0) \Subset V$, then
\begin{equation}
\label{eqn-bound1}
\sup_{B_{\rho y_0/2}(P_0)} (u - \psi)  \leq  Cy_0\rho^2\| \psi \|_{C^{1,1}(\bar B_{\rho_0 y_0}(P_0))}, \quad \forall \rho < \rho_0.
\end{equation}
\end{prop}

\begin{rem}
\label{rem-sup1}
We shall establish \eqref{eqn-bound1} with the aid of certain auxiliary functions, $\zeta$ in \eqref{eqn-zeta0} and $w$ in \eqref{eq:defn-w}, defined on balls $\bar B_{\rho y_0}(P_0)$. (See Fig. \ref{fig:c11_growth_away_degen}.)
\end{rem}

\begin{figure}[htbp]
\centering
\begin{picture}(200,200)(0,0)
\put(0,0){\includegraphics[height=200pt]{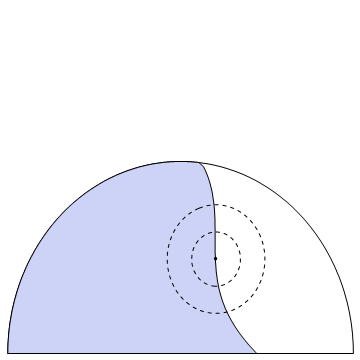}}
\put(45,95){$\scriptstyle V$}
\put(30,20){$\scriptstyle \sE(u)$}
\put(148,20){$\scriptstyle \sC(u)$}
\put(109,53){$\scriptstyle P_0$}
\put(125,71){$\scriptstyle \rho y_0$}
\put(148,53){$\scriptstyle \rho_0 y_0$}
\end{picture}
\caption{Regions in the proof of Proposition \ref{prop-sup1} for estimating the growth of a solution near the free boundary and away from the degenerate boundary.}
\label{fig:c11_growth_away_degen}
\end{figure}

We begin by observing that since $B_{\rho_0 y_0}(P_0) \Subset V$ by assumption, the operator $L$ is uniformly elliptic  on $B_{\rho_0 y_0}(P_0) \Subset \HH$. Consider the linear approximation,
\begin{equation}
\label{eq:ObstacleLinearApprox}
\lo(x, y) := \psi(P_0) + D\psi (P_0) \cdot (x-x_0,y-y_0), \quad (x,y) \in \R^2,
\end{equation}
to our obstacle function $\psi$ at $P_0$. A direct calculation shows that
\begin{equation}
\label{eqn-lo2}
|L (\lo)| \leq M \quad \mbox{on }  B_{\rho_0 y_0}(P_0),
\end{equation}
where, noting that $0<\rho_0<1$ and $0<y_0<1$ as in the hypotheses of Proposition \ref{prop-sup1} and that $\| \psi \|_{C^1(\bar B_{\rho_0 y_0}(P_0))} \leq \| \psi \|_{C^{1,1}(\bar B_{\rho_0 y_0}(P_0))}$,
\begin{equation}
\label{eqn-defnM}
M := K\| \psi \|_{C^{1,1}(\bar B_{\rho_0 y_0}(P_0))},
\end{equation}
and the constant $K>0$ depends at most on the coefficients of $L$. For $0<\rho<\rho_0$, let $\zeta \in C^{2,\alpha}(\bar{B}_{\rho y_0}(P_0))$ be the unique solution (assured by \cite[Theorem 6.14]{GilbargTrudinger}) to the elliptic  boundary value problem,
\begin{equation}
\label{eqn-zeta0}
\begin{cases}
L\zeta = L(\lo) & \mbox{on } \Boy,
\\
\zeta = 10 My_0 \rho^2 & \mbox{on } \partial \Boy.
\end{cases}
\end{equation}
The next lemma provides sharp bounds from above and below on $\zeta$ in terms of $\rho$ and the constant $M$ in \eqref{eqn-zeta0}.

\begin{lem}[Quadratic growth of an auxiliary function near free boundary and away from degenerate boundary]
\label{lem-zeta}
The function $\zeta \in C^{2,\alpha}(\bar{B}_{\rho y_0}(P_0))$ in \eqref{eqn-zeta0} satisfies the bound,
\begin{equation}
\label{eqn-zeta-bound}
M y_0   \rho^2 \leq \zeta  \leq 14  M y_0  \rho^2 \quad \mbox{\emph{on} }\Boy, \quad 0 < \rho < \rho_0,
\end{equation}
where $M$ is as in \eqref{eqn-defnM} and $\rho_0<1$ is a constant depending at most on the coefficients of $L$.
\end{lem}

Before proceeding to the proof of Lemma \ref{lem-zeta}, we consider the effect of rescaling on the operator $L$. Observe that, for any $u\in C^2(\HH)$, if
$$
v(x,y) := u(x_0 + y_0x, y_0 + y_0y), \qquad (\bar x, \bar y) := (x_0 + y_0x, y_0 + y_0y),
$$
then
\begin{align*}
(Lu)(\bar x, \bar y)
&= \frac{y_0 + y_0 y}{2}y_0^{-2}\left(v_{xx} + 2\varrho\sigma v_{xy} + \sigma^2 v_{yy}\right)(x,y)
+ \left(r-q - \frac{y_0(1+y)}{2}\right) y_0^{-1}v_x(x,y)
\\
&\quad + \kappa\left(\theta - y_0(1+y)\right)v_y(x,y) - r v(x,y),
\end{align*}
and therefore,
\begin{equation}
\label{eq:LuLvidentity}
\begin{aligned}
y_0(Lu)(\bar x, \bar y) = (L_{y_0}v)(x,y), \quad \forall (x,y)\in\HH,
\end{aligned}
\end{equation}
where
\begin{equation}
\label{eqn-Lv}
\begin{aligned}
(L_{y_0}v)(x,y) &:= \frac{1+y}{2}\left(v_{xx} + 2\varrho\sigma v_{xy} + \sigma^2 v_{yy}\right)(x,y)
+ \left(r-q - \frac{y_0(1+y)}{2}\right)v_x(x,y)
\\
&\qquad + \kappa\left(\theta - y_0(1+y)\right)v_y(x,y) - ry_0 v(x,y), \quad \forall (x,y)\in\HH.
\end{aligned}
\end{equation}
We now proceed to the

\begin{proof}[Proof of Lemma \ref{lem-zeta}]
Since the ellipticity constant for $L$ depends on $y_0$, we shall use the rescaling in \eqref{eq:LuLvidentity}. Note that the operator $L_{y_0}$ is uniformly elliptic on $B_{1/2}$, since
\begin{equation}
\label{eqn-ell}
  \frac 14 <    \frac{1+y}2 < \frac 34 \quad \mbox{on }\ B_{1/2},
\end{equation}
and the coefficients of $L_{y_0}$ are bounded by a constant (recall that $y_0 <1$) depending at most on the coefficients of $L$. Let
\begin{equation}
\label{eq:defn-bzeta}
\bar \zeta(x,y) := \frac 1{M y_0} \zeta(x_0 + y_0x, y_0 + y_0y), \quad\forall (x,y) \in B_\rho,
\end{equation}
with $\zeta$ as in \eqref{eqn-zeta0}. It follows from \eqref{eqn-lo2} and \eqref{eqn-zeta0} that  $\bar \zeta$ satisfies
\begin{equation}
\label{eqn-bzeta}
\begin{aligned}
|L_{y_0}\bar\zeta| &\leq 1 \quad\hbox{on }\Boo,
\\
\bar\zeta &= 10 \rho^2 \quad \mbox{on } \partial \Boo,
\end{aligned}
\end{equation}
since \eqref{eq:LuLvidentity} yields
$$
(L_{y_0}\bar\zeta)(x,y) = \frac{1}{M y_0}y_0(L\zeta)(\bar x, \bar y) = \frac{1}{M}(L\zeta)(\bar x, \bar y).
$$
We will show that
\begin{equation}
\label{eqn-bzeta-bound}
\rho^2 \leq \bar \zeta  \leq 14    \rho^2  \quad \mbox{on } \Boo,
\end{equation}
provided $\rho <\rho_0$, with $\rho_0<1$ a constant depending at most on the coefficients of $L$, and this will conclude the proof of the lemma, since \eqref{eqn-zeta-bound} follows from \eqref{eq:defn-bzeta} and \eqref{eqn-bzeta-bound}.

To this end, we  consider the barrier function,
\begin{equation}
\label{eq:VarthetaBarrier}
\vartheta(x,y) := ax^2, \quad\forall (x,y)\in\R^2,
\end{equation}
for different choices  of constants $a \in \R$ and  compute that
$$
(L_{y_0}\vartheta)(x,y) = a\left[  (1+y) +  2\left(r - q- \frac {y_0(1+y)}2\right)x -  r  y_0x^2 \right].
$$
Since $1/2 < 1+y < 3/2$ on $B_{1/2}$,  by choosing $\rho < \rho_0$, with $\rho_0<1$ a constant depending at most on the coefficients $r,q$ of $L$,
and using $(x,y) \in B_\rho$ and recalling that $0 < y_0 < 1$,  we can ensure that
\begin{equation}
\label{eqn-theta}
L_{y_0}\vartheta <  \frac 14 \,  a \quad \mbox{on } \Boo, \quad \mbox{if}\,\, a < 0
\quad \mbox{and}  \quad
L_{y_0}\vartheta >  \frac 14 \,  a \quad \mbox{on } \Boo, \quad \mbox{if}\,\, a > 0.
\end{equation}
Choose $a = - 8$ and set $w:=\bar \zeta + \vartheta -  \rho^2$. By combining \eqref{eqn-bzeta} and \eqref{eqn-theta} and using the definition \eqref{eqn-Lv} of $L_{y_0}$, we obtain
$$
L_{y_0}w \leq 1 - 2  + r   \rho^2 y_0<0 \quad \mbox{on } B_\rho,
$$
if $\rho^2 < 1/r$  (remember that $y_0 < 1$). On the other hand,  since $\bar \zeta = 10\rho^2$ on $\partial B_\rho$
by \eqref{eqn-bzeta} and using \eqref{eq:VarthetaBarrier}, we see that
\begin{align*}
w &= \bar \zeta + \vartheta - \rho^2
\\
&\geq 10\rho^2 - 8\rho^2 -  \rho^2 >0 \quad \mbox{on } \partial B_\rho.
\end{align*}
Therefore, the weak maximum principle for $L_{y_0}$ on $B_\rho$ implies that
$$
\bar \zeta +  \vartheta -  \rho^2 \geq 0 \quad\hbox{on } B_\rho.
$$
Since $\vartheta = -8x^2 \leq 0$ on $\R^2$, we conclude that $\bar \zeta \geq   \rho^2 - \vartheta \geq   \rho^2$ in $\Boo$.

We will now estimate $\bar \zeta$ from above. This time we take $a= 4$ and setting $z:=\bar \zeta + \vartheta$,
we now find from \eqref{eqn-bzeta} and \eqref{eqn-theta} that
\begin{equation*}
L_{y_0}z > -1 + \frac{1}{4}4 =0 \quad\hbox{on } B_\rho.
\end{equation*}
But \eqref{eqn-bzeta} and the definition \eqref{eq:VarthetaBarrier} give $\vartheta = 4x^2 \leq 4\rho^2$ on $B_\rho$ and
$$
z=\bar \zeta+ \vartheta \leq 10\rho^2 + 4\rho^2 = 14\rho^2 \quad\hbox{on }\partial \Boo,
$$
and so the weak maximum principle for $L_{y_0}$ on $B_\rho$ shows that $\bar \zeta \leq 14 \rho^2$ on $\Boo$. This finishes the proof of \eqref{eqn-bzeta-bound}, and hence concludes the proof of our lemma.
\end{proof}

\begin{proof}[Proof of Proposition \ref{prop-sup1}]  We shall follow the proof of Lemma 2  in \cite{Caffarelli_jfa_1998}. Our case is more difficult since linear functions are not solutions to the equation $Lu=0$. In addition, our  operator $L$ has variable coefficients and our scaling depends on the ellipticity constant of the operator $L$ on $V$, which is comparable to $y_0$.

With $\lo$ given by \eqref{eq:ObstacleLinearApprox} and $0<\rho<\rho_0$ and $\zeta \in C^{2,\alpha}(\bar B_{\rho y_0}(P_0))$ the function defined by \eqref{eqn-zeta0}, we set
\begin{equation}
\label{eq:defn-w}
w:=u-\lo + \zeta \in H^2(B_{\rho y_0}(P_0))\cap C(\bar B_{\rho y_0}(P_0))
\end{equation}
and observe that
$$
w =  (u-\psi) + (\psi  - \lo + \zeta) \geq \psi  - \lo + \zeta \quad\hbox{on }\bar B_{\rho y_0}(P_0),
$$
since $u\geq \psi$ on $\bar B_{\rho y_0}(P_0)$. By Taylor's theorem,
\begin{equation}
\label{eqn-taylor1}
|\psi(x,y) - \lo (x,y)| \leq 2y_0^2\rho^2 \| \psi \|_{C^{1,1}(\bar B_{\rho y_0}(P_0))}, \quad \forall (x,y) \in \bar B_{\rho y_0}(P_0).
\end{equation}
Since $\zeta \geq M y_0 \rho^2$ by \eqref{eqn-zeta-bound} and $y_0 <1$ (by hypothesis in Proposition \ref{prop-sup1}), we conclude that
\begin{align*}
w  &\geq  - 2y_0^2\rho^2\| \psi \|_{C^{1,1}(\bar B_{\rho y_0}(P_0))} + M y_0\rho^2
\\
&\geq  - 2y_0^2\rho^2\| \psi \|_{C^{1,1}(\bar B_{\rho_0 y_0}(P_0))} + M y_0\rho^2 \quad\hbox{(since $\rho<\rho_0$)}
\\
&>0 \quad \mbox{on } \bar B_{\rho y_0}(P_0)  \quad\hbox{(by definition of $M$)},
\end{align*}
provided that the constant $K$ in the definition \eqref{eqn-defnM} of $M$ is chosen large enough that $K>2$. Also, since  $L\zeta=L(\lo)$ by \eqref{eqn-zeta0}, we have
$$
Lw  = Lu \leq 0 \quad \mbox{a.e. on } \Boy,
$$
where the inequality follows from \eqref{eqn-obst3}. Let us now split $w$ as
\begin{equation}
\label{eq:split-w}
w=w_1 + w_2,
\end{equation}
where $w_1 \in C^{2,\alpha}(\Boy)\cap C(\bar B_{\rho y_0}(P_0))$ is the unique solution (assured by \cite[Theorem 6.13]{GilbargTrudinger}) to
\begin{equation}
\label{eq:BVPw1}
\begin{cases}
L w_1=0  &\mbox{on } \Boy,
\\
w_1=w &\mbox{on } \partial\Boy.
\end{cases}
\end{equation}
(Note that $w=u-\lo + \zeta$ belongs to $C(\partial\Boy)$.)
Because
$$
L(w_1-w)\geq 0 \quad\hbox{a.e. on }\Boy \quad\hbox{and}\quad w_1-w=0 \quad\mbox{on } \partial\Boy,
$$
the weak maximum principle \cite[Theorem 9.1]{GilbargTrudinger} implies
$$
w_1 \leq w \quad\hbox{on }\bar B_{\rho y_0}(P_0),
$$
and, noting that $w >0$ on $\bar B_{\rho y_0}(P_0)$ and thus $w_1=w>0$ on $\partial\Boy$,
$$
w_1 \geq 0 \quad\hbox{on }\bar B_{\rho y_0}(P_0),
$$
so that
\begin{equation}
\label{eq:w1-inequality}
0 \leq w_1 \leq w \quad\hbox{on }\bar B_{\rho y_0}(P_0),
\end{equation}
and hence
\begin{equation}
\label{eq:w2-inequality}
0 \leq w_2 \leq w \quad\hbox{on }\bar B_{\rho y_0}(P_0).
\end{equation}
The inequality \eqref{eq:w1-inequality} obeyed by $w_1$ and the definition \eqref{eq:defn-w} of $w$ yield,
$$
w_1(P_0) \leq  w(P_0) = \zeta(P_0),
$$
and thus, by \eqref{eqn-zeta-bound},
\begin{equation}
\label{eq:w1P0bound}
w_1(P_0) \leq  14M y_0\rho^2.
\end{equation}
Consider the rescaled solution,
\begin{equation}
\label{eq:defn-rescaledw1}
\bar w_1(x,y) := w_1(x_0 + y_0x, y_0 + y_0y), \quad\forall (x,y) \in \bar B_\rho,
\end{equation}
and observe that the function $w_1\in C^{2,\alpha}(B_\rho)\cap C(\bar B_\rho)$, by \eqref{eq:LuLvidentity} and \eqref{eq:BVPw1}, satisfies the uniformly  elliptic  equation,
$$
L_{y_0} \bar w_1=0 \quad\hbox{on } B_\rho.
$$
The Harnack inequality \cite[Corollary 9.25 \& Equation (9.47)]{GilbargTrudinger}, the definition \eqref{eq:defn-rescaledw1} of $\bar w_1$, and the inequality \eqref{eq:w1P0bound} imply the estimate,
$$
\sup_{B_{\rho/2}} \bar  w_1 \leq C'\inf_{B_{\rho/2}} \bar  w_1 \leq C'\bar w_1(0) = C'w_1(P_0) \leq CMy_0\rho^2,
$$
for constants $C'$ and $C=14C'$ which depend at most on the coefficients of $L$, but are independent of $y_0$, and the constant $M$ is given by \eqref{eqn-defnM}. Hence, by \eqref{eq:defn-rescaledw1},
\begin{equation}
\label{eq:sup-bound-w1}
\sup_{B_{\rho y_0/2}(P_0)} w_1 \leq  CMy_0\rho^2, \quad 0 < \rho < \rho_0.
\end{equation}
We will next bound  $w_2$ on $B_{\rho y_0}(P_0)$, taking care to note that (like the regularity of $u$ in Theorem \ref{thm-c11}) $w_2$ only belongs to $H^2(B_{\rho y_0}(P_0))\cap C(\bar B_{\rho y_0}(P_0))$. Recall that $0 \leq w_2 \leq w$ on $\Boy$ by \eqref{eq:w2-inequality} and that $w_2=0$ on  $\partial \Boy$ by \eqref{eq:split-w} and \eqref{eq:BVPw1}. Assume that  $P_1=(x_1,y_1)$ is a  maximum point for the function  $w_2$ on  the closure of the ball $B_{\rho y_0}(P_0)$ and that $w_2(P_1) >0$. Then, $P_1 \in \Boy$ and we consider two cases.

\begin{case}[$P_1 \in \sE(u)$]
If  $P_1 \in \sE(u)$ (where $ u = \psi$), then  $u(P_1)=\psi(P_1)$ and hence, by the inequalities  \eqref{eqn-zeta-bound}, \eqref{eqn-taylor1}, \eqref{eq:w2-inequality}, and definition \eqref{eq:defn-w} of $w$,  we have
$$
w_2(P_1) \leq  w(P_1) =\psi(P_1) - \lo(P_1) + \zeta(P_1) \leq 16 M y_0  \rho^2,
$$
provided the constant $K$ in the definition \eqref{eqn-defnM} of $M$ is chosen large enough that $K>2$.
\end{case}

\begin{case}[$P_1 \in \sC(u)$]
If $P_1 \in \sC(u) $ (where $ u > \psi$) then, since $Lw_2=0$ on the open set $\sC(u) \cap \Boy$ and $w_2$ achieves an interior maximum there, the strong maximum principle \cite[Theorem 3.5]{GilbargTrudinger} implies that $w_2$ must be constant on the connected component of $\sC(u) \cap \Boy$ containing $P_1$. Since $w_2 =0$ on $\partial\Boy$
and $w_2(P_1) >0$ by assumption, it follows that  $w_2(P_1) = w_2(P_2)$ for some point $P_2 \in \sE(u) \cap \Boy$. (Recall that, by hypothesis, $P_0\in\sF(u)$ and so $\sE(u) \cap \Boy$ is non-empty.) Thus, by the inequalities  \eqref{eqn-zeta-bound}, \eqref{eqn-taylor1}, \eqref{eq:w2-inequality}, and definition \eqref{eq:defn-w} of $w$, we have
$$
w_2(P_1) =w_2(P_2) \leq  w(P_2) =\psi(P_2) - \lo(P_2) + \zeta(P_2) \leq 16 M y_0  \rho^2,
$$
provided the constant $K$ in the definition \eqref{eqn-defnM} of $M$ is chosen large enough that $K>2$.
\end{case}

By combining the two cases and recalling that $w_2\leq w_2(P_1)$ on $B_{\rho y_0}(P_0)$, we obtain
\begin{equation}
\label{eq:sup-bound-w2}
\sup_{B_{\rho y_0/2}(P_0)} w_2\leq 16 M y_0\rho^2, \quad 0 < \rho < \rho_0.
\end{equation}
By combining the supremum bounds \eqref{eq:sup-bound-w1} and \eqref{eq:sup-bound-w2} for $w_1$ and $w_2$, respectively, we obtain
\begin{equation}
\label{eqn-wb1}
w \leq CM y_0\rho^2 \quad \mbox{on } B_{\rho y_0/2}(P_0), \quad 0 < \rho < \rho_0,
\end{equation}
where $C$ depends at most on the coefficients of $L$, and $M$ is given by \eqref{eqn-defnM}. This shows, in particular, by \eqref{eqn-zeta-bound} and \eqref{eq:defn-w}, that
$$
u -\lo \leq CMy_0 \rho^2 \quad \mbox{on } B_{\rho y_0/2}(P_0),
$$
where $C$ depends at most on the coefficients of $L$. Now, again using \eqref{eqn-taylor1}, we have
$$
u-\psi = u-\lo + \lo  -  \psi \leq CMy_0\rho^2 \quad \mbox{on } B_{\rho y_0/2}(P_0), \quad 0 < \rho < \rho_0,
$$
for a possibly larger constant $C$ that depends at most on the coefficients of $L$, and this gives the desired bound \eqref{eqn-bound1}.
\end{proof}

We will next establish a supremum bound for the solution, $u$, which holds near $y=0$ and is independent of the $y_0$ coordinate of the point $P_0$.

\begin{prop}[Linear growth of solution near free and degenerate boundaries]
\label{prop-sup2}
Let $u$ be as in Theorem \ref{thm-c11} and let $P_0=(x_0,y_0) \in \F \cap V$ with $0 \leq  y_0 <\theta /4$, where $\theta>0$ is a coefficient of $L$ in \eqref{eqn-oper}. Then, there are a constant $0 < \rho_0 <1$ and  a constant $0<C < \infty$, depending at most on the coefficients of $L$, such that if $B_{\rho_0}^+(P_0) \Subset V\cup\Gamma_0$, then
\begin{equation}
\label{eqn-bound212}
\sup_{B_{\rho/2}^+(P_0)} (u - \psi) \leq  C\rho\| \psi \|_{C^{1,1}(\bar B^+_{\rho_0}(P_0))}, \quad  0 < \rho < \rho_0.
\end{equation}
\end{prop}

\begin{figure}[htbp]
\centering
\begin{picture}(200,200)(0,0)
\put(0,0){\includegraphics[height=200pt]{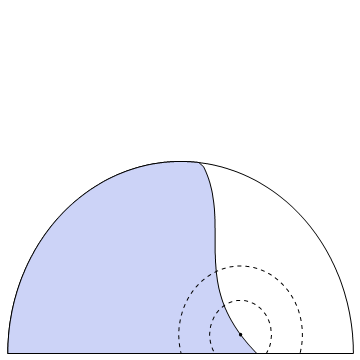}}
\put(45,95){$\scriptstyle V$}
\put(30,20){$\scriptstyle \sE(u)$}
\put(135,60){$\scriptstyle \sC(u)$}
\put(122,12){$\scriptstyle P_0$}
\put(150,25){$\scriptstyle \rho$}
\put(167,30){$\scriptstyle \rho_0$}
\end{picture}
\caption{Regions in the proof of Proposition \ref{prop-sup2} for estimating the growth of a solution near the free boundary and near the degenerate boundary.}
\label{fig:c11_growth_near_degen}
\end{figure}
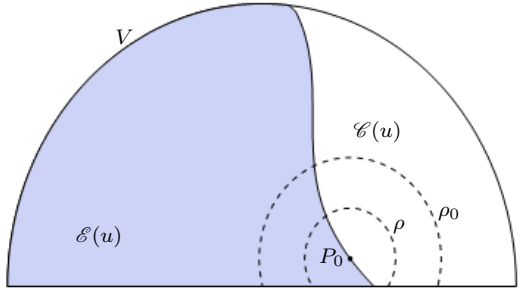

Our proof of Proposition \ref{prop-sup2} follows the pattern of the proof of Proposition \ref{prop-sup1}. (See Fig. \ref{fig:c11_growth_near_degen}.) However, we shall use a different scaling. Observe that
\begin{equation}
\label{eq:LpsiBound}
|L \psi| \leq N\kappa \theta  \quad \mbox{on } B^+_{\rho_0}(P_0),
\end{equation}
where $\kappa>0, \theta>0$ are coefficients of $L$ in \eqref{eqn-oper} and
\begin{equation}
\label{eq:LpsiConstantDefn}
N := K'\| \psi \|_{C^{1,1}(\bar B^+_{\rho_0}(P_0))},
\end{equation}
where $K'$ is a constant which depends at most on the coefficients of $L$ in \eqref{eqn-oper} (remember that $0 \leq  y_0 <\theta /4$ and $0 < \rho_0 <1$).

For $0<\rho<\rho_0$, let $\xi \in C^{2+\alpha}_s(\Bop\cup\Gamma_0) \cap C_b(B^+_\rho(P_0)\cup\Gamma_1)$ be the solution to the boundary value problem,
\begin{equation}
\label{eqn-xi}
\begin{cases}
L\xi = L\psi &\mbox{on } \Bop,
\\
\xi =  10 N  \rho &\mbox{on } \HH\cap\partial \Bo,
\end{cases}
\end{equation}
provided by Theorem \ref{thm-dp-C0BoundaryData-existence}.

\begin{lem}[Linear growth of an auxiliary function near free and degenerate boundaries]
\label{lem-xi}
The function $\xi$ given by \eqref{eqn-xi} satisfies the bound,
\begin{equation}
\label{eqn-xi-bound}
N   \rho  \leq \xi  \leq 20 N  \rho \quad \mbox{\emph{on} }\Bop, \quad 0 < \rho <\rho_0,
\end{equation}
where $N$ is as in \eqref{eq:LpsiBound} and $\rho_0 <1$ is a constant depending at most on the coefficients of $L$.
\end{lem}

\begin{proof}
We first establish the bound from above. We set
$$
z:=\xi + 2 N (y-y_0+\rho)-20N\rho \in C^{2+\alpha}_s(\Bop\cup\Gamma_0) \cap C(\bar B^+_\rho(P_0)),
$$
and use \eqref{eqn-oper} to compute that
\begin{align*}
Lz &= L\xi + 2 N \kappa (\theta -y) - 2  N r  (y-y_0+\rho) + 20N r \rho
\\
&\geq  - N \kappa \theta + 2 N \kappa (\theta -y) - 2  N r  (y-y_0+\rho) + 20N r \rho \quad\hbox{(by \eqref{eq:LpsiBound} and \eqref{eqn-xi})}
\\
&\geq  - N \kappa \theta + N \kappa\theta + 16N r \rho
\\
&\geq 0 \quad\hbox{on }\Bop,
\end{align*}
if $0 \leq y_0 \leq \theta /4$ and $\rho \leq \rho_0$ with $\rho_0 \leq  \min \{ \theta/4,1 \}$ and noting that $0\leq y < y_0+\rho \leq \theta/2$. On the other hand, since $\xi=10 N \rho$ on $\HH\cap\partial \Bop$ by \eqref{eqn-xi}, we have
\begin{align*}
z&=\xi + 2N(y-y_0+\rho) - 20N\rho
\\
&\leq  10 N \rho + 4 N \rho -20N\rho
\\
&\leq  0 \quad\hbox{on } \HH\cap\partial \Bop.
\end{align*}
Hence, the weak  maximum principle for $L$ on $\Bop$ (Theorem \ref{thm-wmp}), implies that $z \leq 0$ on $\Bop$, which implies the desired upper bound in \eqref{eqn-xi-bound},
$$
\xi = z - 2N(y-y_0+\rho) + 20N\rho \leq 20 N\rho \quad\hbox{on } \Bop,
$$
since $y-y_0+\rho \geq -\rho+\rho = 0$ on $\Bop$.

For the bound from below, we now set
$$
z:=\xi - 4  N  (y-y_0+\rho) - N \rho,
$$
and use \eqref{eqn-oper} to compute that
\begin{align*}
Lz &= L\xi - 4 N \kappa (\theta -y)  + 4  N r  (y-y_0+\rho) + N r \rho
\\
&\leq   N \kappa \theta - 4 N \kappa (\theta -y)  + 4  N r  (y-y_0+\rho) + N r \rho \quad\hbox{(by \eqref{eq:LpsiBound} and \eqref{eqn-xi})}
\\
&\leq   N \kappa \theta - 2N \kappa\theta  + 9N r\rho
\\
&\leq   N \kappa \theta    -  2N \kappa \theta +  N \kappa \theta
\\
&\leq 0 \quad\hbox{on }  \Bop,
\end{align*}
if $0 \leq y_0 \leq \theta/4 $ and $\rho \leq \rho_0$ with $\rho_0 \leq \min \{ \theta/4,\kappa\theta/(9r),1\} $ and noting that $0\leq y < y_0+\rho \leq \theta/2$.  On the other hand, since $\xi=10 N \rho$ on $\HH\cap\partial \Bop$ by \eqref{eqn-xi}, we have
\begin{align*}
z&=\xi - 4N(y-y_0+\rho)- N\rho
\\
&\geq  10 N \rho - 8 N \rho  - N \rho
\\
&\geq  0 \quad\hbox{on } \HH\cap\partial \Bop.
\end{align*}
The weak  maximum principle for $L$ on $\Bop$ (Theorem \ref{thm-wmp}) once more shows that $z \geq  0$ on  $\Bop$.  We conclude that
$$
\xi = z + 4 N(y-y_0+\rho) + N \rho \geq N \rho \quad \mbox{on } \Bop,
$$
provided that $\rho < \rho_0$, with $\rho_0 <1$ depending at most on the coefficients of $L$. This yields the desired upper bound in \eqref{eqn-xi-bound} and finishes  the proof of the lemma.
\end{proof}

We will now give the proof of Proposition \ref{prop-sup2}.

\begin{proof}[Proof of Proposition \ref{prop-sup2}]
We give an argument which is similar to the one used in the proof of Proposition \ref{prop-sup1} but we scale our estimate differently and use Lemma \ref{lem-xi} instead of Lemma \ref{lem-zeta}. We set
\begin{equation}
\label{eq:sup2-wdefn}
w:=u - \psi + \xi \in H^2(\Bop,\fw) \cap C(\bar B^+_\rho(P_0)),
\end{equation}
with $\xi \in C^{2+\alpha}_s(\Bop\cup\Gamma_0) \cap C_b(B^+_\rho(P_0)\cup\Gamma_1)$ given by \eqref{eqn-xi}. Then $w$ satisfies
\begin{equation}
\label{eq:Lw-eqn}
Lw  = Lu \quad \mbox{a.e. on } \Bop.
\end{equation}
Let us now split $w$ as  $w=w_1 + w_2$, where $w_1 \in C^{2+\alpha}_s(\Bop\cup\Gamma_0) \cap C_b(B^+_\rho(P_0)\cup\Gamma_1)$ (whose existence is assured by Theorem \ref{thm-dp-C0BoundaryData-existence}) is defined by
\begin{equation}
\label{eq:sup2-w1defn}
\begin{cases}
L w_1 =0 &\mbox{on } \Bop,
\\
w_1 =w &\mbox{on } \HH\cap\partial\Bop.
\end{cases}
\end{equation}
By the weak maximum principle for $L$ on $\Bop$ (Theorem \ref{thm-wmp}) and the fact that  $w \geq \xi \geq 0$ on $\Bop$, we have
\begin{equation}
\label{eq:sup2-w1bounds}
0 \leq w_1 \leq w \quad \mbox{on } \Bop,
\end{equation}
and thus
\begin{equation}
\label{eq:sup2-w2bounds}
0 \leq w_2 \leq w \quad \mbox{on } \Bop.
\end{equation}
From \eqref{eqn-xi-bound}, \eqref{eq:sup2-wdefn}, \eqref{eq:sup2-w1bounds}, and the fact that $u(P_0)=\psi(P_0)$, we see that
\begin{equation}
\label{eq:sup2-w1P0growth}
w_1(P_0) \leq  w(P_0) = \xi(P_0) \leq 20N \rho \quad \mbox{on } \Bop.
\end{equation}
Set $(x_\rho,y_\rho):=(x_0+\rho y, y_0+\rho y)$ and consider the rescaled solution,
$$
\bar w_1(x,y) := w_1(x_0 + \rho x, y_0 + \rho y), \quad (x,y) \in B_1 \cap \{ y_\rho \geq 0 \},
$$
which satisfies the equation,
\begin{equation}
\label{eq:Lrhobarw1-eqn}
L_\rho \bar w_1=0 \quad\hbox{on } B_1 \cap \{ y_\rho >0 \},
\end{equation}
where (compare \eqref{eqn-Lv})
\begin{equation}
\label{eqn-Lr}
\begin{aligned}
L_\rho  v &:= \frac{y_\rho}{2\rho}\left(v_{xx} + 2\varrho\sigma v_{xy} + \sigma^2v_{yy}\right) + (r- q- y_\rho)v_x
\\
&\qquad + \kappa (\theta - y_\rho)v_y - r \rho v, \quad\forall v\in C^\infty(\HH),
\end{aligned}
\end{equation}
and using the fact that (compare \eqref{eq:LuLvidentity})
$$
\rho (Lw_1)(x_\rho,y_\rho) = (L_\rho \bar w_1)(x,y).
$$
From \eqref{eq:Lrhobarw1-eqn}, the Harnack inequality\footnote{See also \cite[Theorem 4.5.3]{Koch} for a version of the Harnack inequality for the linearization of the parabolic porous medium equation.} \cite[Theorem 1.16]{Feehan_Pop_regularityweaksoln} yields the estimate
$$
\sup_{B_{1/2}(0) \cap \{ y_\rho > 0 \}} \bar w_1 \leq C\inf_{B_{1/2}(0) \cap \{ y_\rho > 0 \}} \bar w_1 \leq C \bar w(0),
$$
for a  constant, $C$, depending at most on the coefficients of $L$. Combining the preceding inequality with \eqref{eq:sup2-w1P0growth} yields
$$
\sup_{B^+_{\rho/2} (P_0)} w_1\leq C w(P_0) \leq  20 C N \rho,
$$
that is,
\begin{equation}
\label{eq:sup2-bound-w1}
\sup_{B^+_{\rho/2} (P_0)} w_1 \leq  20 C N \rho, \quad 0 < \rho < \rho_0,
\end{equation}
for a constant, $\rho_0$, depending at most on the coefficients of $L$.

We will next bound  $w_2$ on $B_\rho^+(P_0)$, following the same reasoning as in the proof of Proposition \ref{prop-sup1}. Recall that $0 \leq w_2 \leq w$ on $B_\rho^+(P_0)$ by \eqref{eq:sup2-w2bounds} and $w_2=0$  on $\HH\cap\partial B_\rho^+(P_0)$ by \eqref{eq:sup2-w1defn}. Assume that $P_1=(x_1,y_1)$ is a maximum point for the function $w_2$ on $\bar B_{\rho}^+(P_0)$ and that $w_2(P_1) >0$. Therefore, $P_1 \in B_{\rho}^+(P_0)\cup \Gamma_0$, where (by our convention) $\Gamma_0 = \{y=0\}\cap\partial B_{\rho}^+(P_0)$.

\setcounter{case}{0}
\begin{case}[$P_1 \in \sE(u)$]
If  $P_1 \in \sE(u)\cap (B_{\rho}^+(P_0)\cup\Gamma_0)$, then  $u(P_1)=\psi(P_1)$. Recalling that $w=u-\psi+\xi$ by \eqref{eq:sup2-wdefn}, we conclude from \eqref{eqn-xi-bound} and \eqref{eq:sup2-w2bounds} that at $P_1$ we have the bound
$$
w_2(P_1) \leq  w(P_1) = \xi(P_1)  \leq 20 N  \rho,
$$
provided $\rho< \rho_0$.
\end{case}

\begin{case}[$P_1 \in \sC(u)$]
If $P_1 \in \sC(u) \cap (B_{\rho}^+(P_0)\cup\Gamma_0)$ (where $ u > \psi$), then since $w_2 \in H^2(\Bop,\fw) \cap C(\bar B^+_\rho(P_0))$ obeys
$$
Lw_2=0 \quad\hbox{a.e. on }\sC(u) \cap B_{\rho}^+(P_0),
$$
by \eqref{eqn-obst3}, \eqref{eq:Lw-eqn}, and \eqref{eq:sup2-w1defn},
the regularity result in Proposition \ref{prop-c2a} implies that $w_2$ also belongs to $C^{2+\alpha}_s(\sC(u) \cap (B_{\rho}^+(P_0)\cup\Gamma_0))$. But $w_2$ achieves a positive maximum at $P_1\in\sC(u) \cap (B_{\rho}^+(P_0)\cup\Gamma_0)$ and so the strong maximum principle (Theorem \ref{thm-smp}) implies that $w_2$ must be constant on the connected component of $\sC(u) \cap (B_{\rho}^+(P_0)\cup\Gamma_0)$ containing $P_1$.  Since $w_2 =0$  on $\HH\cap\partial B_{\rho}^+(P_0)$ and $w_2(P_1) >0$, it follows that  $w_2(P_1) = w_2(P_2)$, for some point $P_2$ with $P_2 \in  \sE(u) \cap (B_{\rho}^+(P_0)\cup\Gamma_0)$. (Recall that, by hypothesis, $P_0\in\sF(u)$ and so $\sE(u) \cap (B_{\rho}^+(P_0)\cup\Gamma_0)$ is non-empty.)  We conclude that by \eqref{eqn-xi-bound}, \eqref{eq:sup2-wdefn}, \eqref{eq:sup2-w2bounds}, and the fact that $u(P_2)=\psi(P_2)$,
$$
w_2(P_1) = w_2(P_2) \leq w(P_2) = \xi(P_2)  \leq 20 N\rho,
$$
provided $\rho< \rho_0$.
\end{case}

Combining the two cases and recalling that $w_2\leq w_2(P_1)$ on $\bar B_{\rho}^+(P_0)$, by definition of $P_1$, yields
\begin{equation}
\label{eq:sup2-bound-w2}
\sup_{B_{\rho}^+(P_0)} w_2\leq 20 N\rho, \quad 0 < \rho < \rho_0.
\end{equation}
Combining the estimates \eqref{eq:sup2-bound-w1} and \eqref{eq:sup2-bound-w2}, respectively, for $w_1$ and $w_2$ yields
$$
\sup_{B^+_{\rho/2} (P_0)} w \leq C'N\rho, \quad 0 < \rho < \rho_0,
$$
for a constant $C'$ which depends at most on the coefficients of $L$, and $N$ is as in \eqref{eq:LpsiConstantDefn}. This yields the desired bound \eqref{eqn-bound212}.
\end{proof}

\begin{cor}[Linear growth of solution near free and degenerate boundaries]
\label{cor-sup3}
Under the hypotheses of Proposition \ref{prop-sup2},  there are a constant $0 < \rho_0 <1$ and a constant $0<C < \infty$, depending at most on the coefficients of $L$, such that
\begin{equation}
\label{eqn-bound3}
\sup_{B_{\rho/2}^+(P_0)} (u - \psi (P_0)) \leq  C\rho\| \psi \|_{C^{1,1}(\bar B^+_{\rho_0}(P_0))}, \quad 0 < \rho < \rho_0.
\end{equation}
\end{cor}

\section{Proof of main theorem}
\label{sec:ProofMainTheorem}
We will establish in this section the $C^{1,1}_s$ regularity of our solution, $u$, in Theorem \ref{thm-c11}. For a much simpler example --- interior $C^{1,1}$ regularity for a solution, $u$, to $\min\{\Delta u-1, u\} = 0$ on a bounded domain in $\R^n$ --- but one which conveys some of the flavor of our proof of Theorem \ref{thm-c11}, see the proof of Theorem 1.1 in \cite[p. 11]{Petrosyan_2011_msrilectures}, which is based in turn on ideas of Caffarelli \cite{Caffarelli_jfa_1998}.

\begin{proof}[Proof of Theorem  \ref{thm-c11}]
Because of Proposition \ref{prop-fzero} we may assume without loss of generality that  $u \in H^2(V,\fw) \cap C(\bar V)$ is  a solution to the \emph{homogeneous} obstacle problem \eqref{eqn-obst_homog}  with obstacle function $\psi \in C_s^{2+\alpha}(\bar V)$ and $f=0$ on $V$. Recall that $V=B_{R_0}^+(Q_0)$ is as in \eqref{eq:domains}, for some $R_0 > 0$ and $Q_0=(p_0,q_0) \in \bar\HH$. We may also assume without loss of generality that
$$
0<R_0\leq 1,
$$
and also that
$$
\Lambda=1 \quad\hbox{and}\quad 0 \leq q_0 \leq 1,
$$
since $L$ is uniformly elliptic on $V=B_{R_0}(Q_0)$ when $q_0 >1$ and standard results imply that $u\in C^{1,1}(\bar V)$ \cite[Theorem 4.38]{Troianiello}.

\setcounter{figure}{0}
\begin{figure}[htbp]
\centering
\begin{picture}(200,200)(0,0)
\put(0,0){\includegraphics[height=200pt]{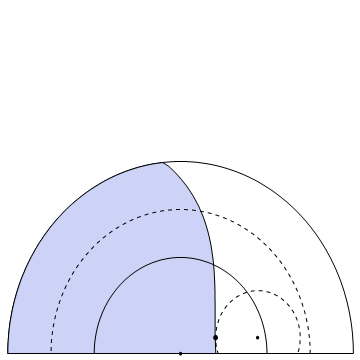}}
\put(45,95){$\scriptstyle V$}
\put(60,77){$\scriptstyle V_1$}
\put(80,55){$\scriptstyle U$}
\put(75,30){$\scriptstyle \sE(u)$}
\put(108,92){$\scriptstyle \sF(u)$}
\put(142,77){$\scriptstyle \sC(u)$}
\put(89,7){$\scriptstyle Q_0$}
\put(109,10){$\scriptstyle P_1$}
\put(132,10){$\scriptstyle P_0$}
\put(145,40){$\scriptstyle dy_0$}
\end{picture}
\caption{Regions in the proof of Theorem \ref{thm-c11} for estimating the $C^{1,1}_s$ norm of a solution.}
\label{fig:c11_estimate_regions}
\end{figure}
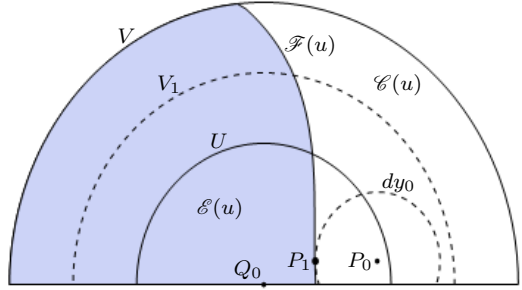

Let $P_0=(x_0,y_0) \in \sC(u) \cap  \bar U$, where $U=B_{R_0/2}^+(Q_0)$ as in \eqref{eq:domains} with $R=R_0/2$; see Fig. \ref{fig:c11_estimate_regions}. Assuming without loss of generality that\footnote{Our assumptions so far that $P_0\in\bar U$ and $R_0=1$ ensure $0\leq y_0\leq 3/2$, but standard results apply when $y_0\geq 1$.}
\begin{equation}
\label{eq:y0leq1}
0 < y_0 \leq 1,
\end{equation}
we  will  establish the bound
\begin{equation}
\label{eqn-c11}
y_0 |D^2 u(P_0)|  + |Du(P_0)| + |u(P_0)| \leq  C\left(\| u \|_{C(\bar V)} + \| \psi \|_{C^{1,1}(\bar V)}\right),
\end{equation}
where  the constant $C=C(L,R_0)$ may depend $R_0$ and  the coefficients of $L$. Since the constant $C$ will not depend on $y_0$ (if $y_0$ obeys \eqref{eq:y0leq1}) this will provide the desired $C^{1,1}_s$ bound on $u$ up to $y=0$. Set
$$
V_1:=B_{3R_0/4}^+(Q_0).
$$
Since $u$ is continuous on $\bar V$, the exercise region, $\E(u)$, as defined in \eqref{eqn-con}, is a relatively closed subset of $V\cup\Gamma_0$. We may suppose without loss of generality that
$$
\E(u)\cap V_1 \neq \emptyset.
$$
Otherwise, $V_1\subset\C(u)$, where $\C(u)$ is the continuation region, as defined in \eqref{eqn-om}, and because $Lu=0$ on $\C(u)$, Theorem \ref{thm-c11} would follow immediately from Proposition \ref{prop-c2a}.
Now let $d$ be the maximum number such that
\begin{equation}
\label{eq:defn-d}
\Qdy \cap \E(u) \cap \bar V_1 = \emptyset.
\end{equation}
Then there exists at least one point\footnote{We alert the reader that in section \ref{sec:SupBounds} we use $P_0$ to denote a point in $\sF(u)$ whereas in this section we use $P_0$ to denote a point in $\sC(u)$ and $P_1$ to denote a point in $\sF(u)$.}
\begin{equation}
\label{eq:defn-P1}
P_1=(x_1,y_1) \in \partial  \Qdy \cap \F \cap \bar V_1.
\end{equation}
Since $P_0\in \bar U$ and $P_1 \in \bar V_1$, we have $0<dy_0 \leq 5R_0/4$.

Throughout this section, we let $0<\rho_0<1$ denote the smaller of the two constants in Propositions \ref{prop-sup1} and \ref{prop-sup2} and, by replacing $\rho_0$ with a smaller constant if needed, we may assume that
\begin{equation}
\label{eq-size-rho0}
0 < \rho_0 < \min\{1,R_0/5\}.
\end{equation}
We shall distinguish between three situations. We begin with the first situation.

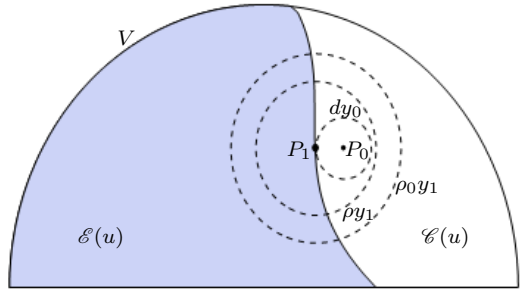
\begin{figure}[htbp]
\centering
\begin{picture}(200,200)(0,0)
\put(0,0){\includegraphics[height=200pt]{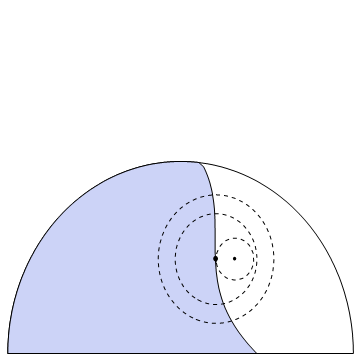}}
\put(45,95){$\scriptstyle V$}
\put(30,20){$\scriptstyle \sE(u)$}
\put(160,20){$\scriptstyle \sC(u)$}
\put(109,53){$\scriptstyle P_1$}
\put(131,53){$\scriptstyle P_0$}
\put(125,69){$\scriptstyle dy_0$}
\put(130,30){$\scriptstyle \rho y_1$}
\put(150,40){$\scriptstyle \rho_0 y_1$}
\end{picture}
\caption{Regions for Step \ref{step:pfthm-d_leq_quarterrhozero} (where $d \leq  {\rho_0}/4$) of the proof of Theorem \ref{thm-c11} for estimating the $C^{1,1}_s$ norm of a solution near the free boundary and away from the degenerate boundary.}
\label{fig:c11_estimate_away_degen}
\end{figure}

\begin{step}[$d \leq  {\rho_0}/4$]
\label{step:pfthm-d_leq_quarterrhozero}
Since $\rho_0 < 1$, we have $B_{\rho_0 y_1}(P_1) \Subset \HH$ (unless $y_1=0$), while  $P_1\in \bar V_1$ implies  $\dist(P_1,\HH\cap\partial V) \geq R_0/4$; since we also have $\rho_0 < R_0/5$ by \eqref{eq-size-rho0}, we may conclude that
\begin{equation}
\label{eqn-b11}
B_{\rho_0 y_1}(P_1) \Subset V \quad \hbox{if }  0<y_1<\frac{5}{4}.
\end{equation}
Since $P_1=(x_1,y_1) \in \partial B_{dy_0}(P_0)$ and $P_0=(x_0,y_0)$, we have $|y_1-y_0| \leq y_0 d < y_0/4$ and thus
\begin{equation}
\label{eq:Comparey0y1}
3y_0 /4 < y_1 < 5y_0 / 4,
\end{equation}
and hence $y_0$ and $y_1$ are comparable. In particular,
we have  $0<y_1<5/4$ by \eqref{eq:y0leq1}, and hence by \eqref{eqn-b11} (see Fig. \ref{fig:c11_estimate_away_degen}), we see that
$$
B_{\rho_0 y_1}(P_1) \Subset V.
$$
Set  $\rho := 4d$  and let $\zeta \in C^{2,\alpha}(\bar{B}_{\rho y_1}(P_1))$ be the function defined  by \eqref{eqn-zeta0} (with $P_0$ and $y_0$ replaced by $P_1$ and $y_1$, respectively), that is,
\begin{equation}
\label{eq:defn-zeta-pfthm}
\begin{cases}
L\zeta = L(l_{P_1}) &\hbox{on } B_{\rho y_1}(P_1),
\\
\zeta = 10 M y_1 \rho^2 &\hbox{on } \partial B_{\rho y_1}(P_1),
\end{cases}
\end{equation}
where (compare \eqref{eqn-defnM})
\begin{equation}
\label{eqn-defnM1}
M := K\| \psi \|_{C^{1,1}(\bar B_{\rho_0 y_1}(P_1))},
\end{equation}
and $K$ is a constant which depends at most on the coefficients of $L$, and the following inequality holds (compare \eqref{eqn-lo2})
\begin{equation}
\label{eq:lo2-P1}
|L\zeta| = |L(l_{P_1})| \leq M \quad\hbox{on }B_{\rho y_1}(P_1).
\end{equation}
It follows from  \eqref{eqn-zeta-bound}  that
\begin{equation}
\label{eqn-bound33}
M y_1 \rho^2 \leq \zeta \leq 14 M y_1 \rho^2 \quad \mbox{on }B_{\rho y_1}(P_1).
\end{equation}
Moreover, since $P_1=(x_1,y_1) \in \partial  \Qdy$ and $3y_0 /4 < y_1$  and $d = \rho/4 <1/4$, we have $\dist(P_1,P_0)=dy_0$ and $2dy_0 \leq 8dy_1/3 = 2\rho y_1/3 \leq \rho y_1$, and thus  (see Fig. \ref{fig:c11_estimate_away_degen})
\begin{equation}
\label{eq:Ballinclusion1}
\Bdy \subset B_{\rho y_1}(P_1).
\end{equation}
Therefore, $\zeta$ is also defined on $\bar B_{dy_0}(P_0)$ and satisfies the bounds \eqref{eqn-bound33} on $B_{dy_0}(P_0)$
with $\rho$ replaced by $d=\rho/4$:
$$
\frac{1}{16}M y_1 d^2 \leq \zeta \leq \frac{14}{16}M y_1 d^2 \quad \mbox{on }B_{dy_0}(P_0),
$$
and thus, applying \eqref{eq:Comparey0y1},
\begin{equation}
\label{eqn-bound33-onP0halfdy1ball}
\frac{3}{64}M y_0 d^2 \leq \zeta \leq \frac{70}{64}M y_0 d^2 \quad \mbox{on }B_{dy_0}(P_0).
\end{equation}
As in the proof of Proposition \ref{prop-sup1}, we set
\begin{equation}
\label{eq:defnw-case1-pfthm}
w:= u - l_{P_1}  + \zeta \in H^2(B_{\rho y_1}(P_1))\cap C(\bar B_{\rho y_1}(P_1)).
\end{equation}
The inequality \eqref{eqn-wb1} (with the role of $B_{\rho y_0/2}(P_0)$ there replaced by $B_{\rho y_1/2}(P_1)$) yields
$$
0 \leq w \leq CMy_1\rho^2 \quad \mbox{on } B_{\rho y_1/2}(P_1),
$$
and thus, since $y_1 \leq 5y_0/4$ by \eqref{eq:Comparey0y1} and $\rho = 4d$ and $B_{dy_0/2}(P_0) \subset B_{\rho y_1/2}(P_1)$ by \eqref{eq:Ballinclusion1},
\begin{equation}
\label{eq:wbounds-pfthm}
0 \leq w \leq CM y_0 d^2 \quad \mbox{on } B_{dy_0/{2}}(P_0),
\end{equation}
for a larger constant $C$ depending at most on the coefficients of $L$ and where $M$ is as in \eqref{eqn-defnM1} (compare \eqref{eqn-defnM}).

Because $\Qdy \subset\sC(u)$, we have $Lu=0$ on $\Qdy$, while $L(l_{P_1}) = L\zeta$ on $B_{\rho y_1}(P_1)$ by \eqref{eq:defn-zeta-pfthm}. It follows that
\begin{equation}
\label{eq:Lw}
Lw=0 \quad\hbox{on } \Qdy,
\end{equation}
since $\Qdy \subset B_{\rho y_1}(P_1)$ by \eqref{eq:Ballinclusion1}. Consider now the rescaled solution, $\bar w \in C^{2,\alpha}(\bar B_d)$, given by
\begin{equation}
\label{eq:defn-bw-pfthm}
\bar w(x,y) := w(x_0 + y_0 x, y_0 + y_0 y), \quad\forall (x,y) \in \bar B_d,
\end{equation}
to the uniformly elliptic equation
\begin{equation}
\label{eq:Lw_rescaled}
L_{y_0} \bar w = 0 \quad\hbox{on } B_d,
\end{equation}
where $B_d = B_d(0,0)$ and $L_{y_0}$ is given by \eqref{eqn-Lv}. The classical Schauder interior estimates for strictly elliptic equations \cite[Corollary 6.3]{GilbargTrudinger} yield
$$
d\|D \bar w\|_{C(\bar B_{d/4})} + d^2\|D^2\bar w\|_{C(\bar B_{d/4})} \leq C\| \bar w \|_{C(\bar B_{d/2})},
$$
for a constant $C$ depending at most on the coefficients of $L$ (noting that $d\leq 1/4$). Combining the preceding inequality with the inequalities \eqref{eq:wbounds-pfthm} for $w$ implies the bounds
$$
\frac 1d  |D \bar w(0)| + |D^2 \bar w(0) | \leq \frac C{d^2} \| \bar w \|_{C(\bar B_{d/2})}  \leq CMy_0,
$$
where $C$ depends at most on the coefficients of $L$, and $M$ is as in \eqref{eqn-defnM1}. Hence, since $D\bar w(0) = y_0Dw(P_0)$ and $D^2\bar w(0) = y_0^2D^2w(P_0)$ by \eqref{eq:defn-bw-pfthm}, we obtain
$$
\frac{y_0}d  |Dw(P_0)| + y_0^2|D^2w(P_0) | \leq CMy_0.
$$
We conclude that
\begin{equation}
\label{eqn-boundw33}
|Dw(P_0)| + y_0 |D^2 w(P_0) |  \leq CM,
\end{equation}
Similarly, the rescaled function $\bar\zeta \in C^{2,\alpha}(\bar B_d)$ given by
\begin{equation}
\label{eq:defn-bzeta-pfthm}
\bar \zeta(x,y) := \zeta (x_0 + y_0 x, y_0 + y_0 y), \quad\forall (x,y) \in \bar B_d
\end{equation}
satisfies the uniformly elliptic equation (see \eqref{eq:LuLvidentity})
$$
L_{y_0} \bar \zeta  = y_0L\zeta = y_0L(l_{P_1}) = y_0 f_1 \quad \mbox{on }B_d,
$$
where $f_1:=L(l_{P_1})$ is a smooth, linear function with
\begin{align*}
\| f_1 \|_{C^1(B_d)} &\leq C\left(|\psi(P_0)| + |D\psi(P_0)|\right)
\\
&\leq C\| \psi \|_{C^{1,1}(\bar B_{dy_0}(P_0))} \leq C\| \psi \|_{C^{1,1}(\bar B_{\rho_0 y_1}(P_1))}
\\
&= CM,
\end{align*}
while $C$ depends at most on the coefficients of $L$, and $M$ is as in \eqref{eqn-defnM1}.
Define $\eta \in C^{2,\alpha}(\bar B_1)$ by
$$
\bar\zeta(x,y) =: d^2\eta(x/d,y/d), \quad (\bar x,\bar y) := (x/d,y/d) \in B_1,
$$
The function $\eta$ obeys
\begin{align*}
(L_{y_0}\bar\zeta)(x,y) &= \frac{1+y}{2}\left(\bar\zeta_{xx} + 2\varrho\sigma\bar\zeta_{xy} + \sigma^2\bar\zeta_{yy}\right)(x,y)
+ \left(r-q - \frac{y_0(1+y)}{2}\right)\bar\zeta_x(x,y)
\\
&\qquad + \kappa\left(\theta - y_0(1+y)\right)\bar\zeta_y(x,y) - ry_0\bar\zeta(x,y)
\\
&= \frac{1+d\bar y}{2}\left(\eta_{\bar x\bar x} + 2\varrho\sigma\eta_{\bar x\bar y} + \sigma^2\eta_{\bar y\bar y}\right)(\bar x,\bar y)
+ d\left(r-q - \frac{y_0(1+d\bar y)}{2}\right)\eta_{\bar x}(\bar x,\bar y)
\\
&\qquad + d\kappa\left(\theta - y_0(1+d\bar y)\right)\eta_{\bar y}(\bar x,\bar y) - ry_0d^2\eta(\bar x,\bar y)
\\
&=: L_{y_0,d}\eta(\bar x,\bar y), \quad \forall (\bar x,\bar y)\in B_1,
\end{align*}
and
$$
L_{y_0,d}\eta(\bar x,\bar y) = y_0f_1(x,y) = y_0f_1(d\bar x,d\bar y) =: y_0\bar f_1(\bar x,\bar y), \quad \forall (\bar x,\bar y)\in B_1.
$$
We have $Df_1(x,y) = d^{-1}D\bar f_1(\bar x,\bar y)$ and so, noting that $0<d\leq 1/4$,
\begin{align*}
\| \bar f_1 \|_{C^1(B_1)} &= \| \bar f_1 \|_{C(\bar B_1)} + \| D\bar f_1 \|_{C(\bar B_1)}
\\
&=  \| f_1 \|_{C(\bar B_d)} + d\| Df_1 \|_{C(\bar B_d)} \leq CM.
\end{align*}
Applying the classical Schauder interior estimates \cite[Corollary 6.3]{GilbargTrudinger} to the solution $\eta$ to $L_{y_0,d}\eta = \bar f_1$ on $B_1$ gives
\begin{align*}
\|D \eta\|_{C(\bar B_{1/2})} + \|D^2\eta\|_{C(\bar B_{1/2})} &\leq C\left(\| \eta\|_{C(\bar B_1)} + y_0\| \bar f_1 \|_{C^1(B_1)}\right)
\\
&\leq C\left(\| \eta\|_{C(\bar B_1)} + My_0\right),
\end{align*}
for a constant $C$ depending at most on the coefficients of $L$ (recall that $0<y_0\leq 1$). Therefore, on $B_d$,
$$
d^{-1}\|D \bar \zeta\|_{C(\bar B_{d/2})} + \|D^2\bar \zeta\|_{C(\bar B_{d/2})} \leq C\left(d^{-2}\| \bar \zeta\|_{C(\bar B_d)} + My_0\right),
$$
for a constant $C$ depending at most on the coefficients of $L$. Combining the preceding inequality with the bound \eqref{eqn-bound33-onP0halfdy1ball} for $\zeta$ yields
\begin{align*}
d^{-1}|D\bar \zeta(0)|+  |D^2\bar \zeta(0)| &\leq  C\left(d^{-2}\| \bar \zeta  \|_{C(\bar B_d)} + My_0\right)
\\
&= C\left(d^{-2}\|\zeta  \|_{C(\bar B_{dy_0}(P_0))} + M y_0\right)
\\
&\leq CM y_0,
\end{align*}
for a larger constant $C$, but depending at most on the coefficients of $L$. Hence, since $D\bar \zeta(0) = y_0D\zeta(P_0)$ and $D^2\bar \zeta(0) = y_0^2D^2\zeta(P_0)$ by \eqref{eq:defn-bzeta-pfthm}, we obtain
$$
d^{-1}y_0|D\zeta(P_0)|+  y_0^2|D^2\zeta(P_0)| \leq CM y_0,
$$
and thus, noting that $4\leq d^{-1}$,
\begin{equation}
\label{eqn-boundz33}
|D\zeta (P_0)| + y_0 |D^2 \zeta(P_0) |  \leq CM.
\end{equation}
Recalling that $w=u-l_{P_1} +\zeta$ by \eqref{eq:defnw-case1-pfthm}, we conclude from \eqref{eqn-boundw33} and \eqref{eqn-boundz33} that
$$
|Du(P_0)| + y_0 |D^2u(P_0) |  \leq CM,
$$
where $M$ is as in \eqref{eqn-defnM1} and so \eqref{eqn-c11} holds for this step.
\end{step}

We consider the second situation.
\begin{figure}[htbp]
\centering
\begin{picture}(200,200)(0,0)
\put(0,0){\includegraphics[height=200pt]{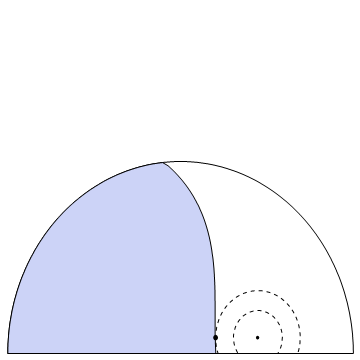}}
\put(45,95){$\scriptstyle V$}
\put(30,20){$\scriptstyle \sE(u)$}
\put(135,69){$\scriptstyle \sC(u)$}
\put(109,10){$\scriptstyle P_1$}
\put(145,10){$\scriptstyle P_0$}
\put(130,40){$\scriptstyle dy_0$}
\put(136,28){$\scriptstyle \rho_0/4$}
\end{picture}
\caption{Regions for Case \ref{case:pfthm-dgeq1_and_dy0geqquarterrhozero} of Step \ref{step:pfthm-dgeq1} (where $d > 1$ and $d y_0 \geq  \rho_0/4$) of the proof of Theorem \ref{thm-c11} for estimating the $C^{1,1}_s$ norm of a solution near the free boundary and near the degenerate boundary.}
\label{fig:c11_estimate_near_degen}
\end{figure}

\begin{step}[$d > 1$]
\label{step:pfthm-dgeq1}
We shall consider two cases.

\setcounter{case}{0}
\begin{case}[$d > 1$ and $d y_0 \geq  \rho_0/4$]
\label{case:pfthm-dgeq1_and_dy0geqquarterrhozero}
Since $d y_0 \geq  \rho_0/4$ for this case
(see Fig. \ref{fig:c11_estimate_near_degen})
$$
B^+_{\rho_0/4}(P_0) \subset B^+_{dy_0}(P_0)\cap V \subset \sC(u)\cap V,
$$
and so $Lu=0$ on $B_{\rho_0/4}^+(P_0)$. The Schauder estimate \eqref{eqn-schauder} therefore yields
\begin{equation}
\label{eq:c11u-schauder}
\| u \|_{C^{1,1}_s(\bar B_{\rho_0/8}^+(P_0))} \leq C\| u \|_{C^{2+\alpha}_s(\bar B_{\rho_0/8}^+(P_0))} \leq C\| u \|_{C(\bar B_{\rho_0/4}^+(P_0))},
\end{equation}
with a constant, $C$, depending at most on $\alpha,\rho_0$, and the coefficients of $L$,
recalling that we have chosen $0<y_0\leq 1$ by our assumption \eqref{eq:y0leq1} for this section (and thus $\Lambda=1$).
This yields  the desired bound  \eqref{eqn-c11}  for this case.
\end{case}

\begin{case}[$d > 1$ and $dy_0 < \rho_0/4$]
\label{case:pfthm-dgeq1_and_dy0lessquarterrhozero}

\begin{figure}[htbp]
\centering
\begin{picture}(200,200)(0,0)
\put(0,0){\includegraphics[height=200pt]{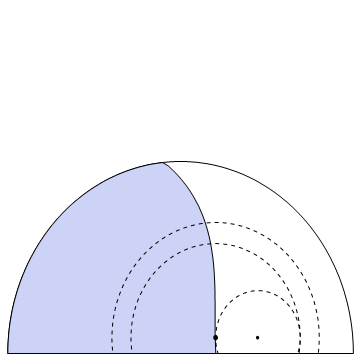}}
\put(45,95){$\scriptstyle V$}
\put(30,20){$\scriptstyle \sE(u)$}
\put(118,83){$\scriptstyle \sC(u)$}
\put(109,10){$\scriptstyle P_1$}
\put(145,10){$\scriptstyle P_0$}
\put(135,40){$\scriptstyle dy_0$}
\put(89,72){$\scriptstyle \rho_0$}
\put(128,64){$\scriptstyle 2dy_0$}
\end{picture}
\caption{Regions for Case \ref{case:pfthm-dgeq1_and_dy0lessquarterrhozero} of Step \ref{step:pfthm-dgeq1} (where $d > 1$ and $dy_0 < \rho_0/4$) of the proof of Theorem \ref{thm-c11} for estimating the $C^{1,1}_s$ norm of a solution near the free boundary and near the degenerate boundary. (The radii $2dy_0$ and $\rho_0$ are not drawn to scale, since $2dy_0 < \rho_0/2$ for this case.)}
\label{fig:c11_estimate_near_degen_smalldyzero}
\end{figure}
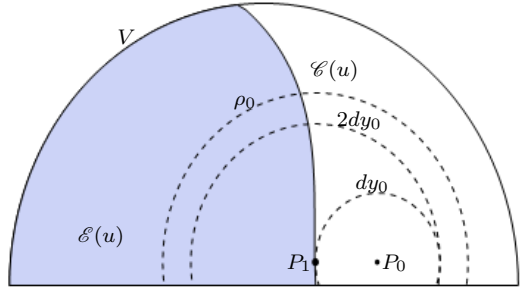

Since $dy_0 < \rho_0/4$, and also $P_1\in\bar V_1$ and $\dist(P_1,P_0)=dy_0$ by \eqref{eq:defn-P1} and $\rho_0 < R_0/2$ by \eqref{eq-size-rho0}, we have (see Fig. \ref{fig:c11_estimate_near_degen_smalldyzero})
$$
B^+_{dy_0}(P_0) \subset B^+_{2dy_0}(P_1)\subset B^+_{\rho_0}(P_1) \Subset V\cup\Gamma_0,
$$
Thus, it follows from \eqref{eqn-bound3} (with $P_0$ replaced by $P_1$ and $\rho = 4dy_0 < \rho_0$) and Taylor's theorem (since $\dist(P_1,P_0)=dy_0$) that
\begin{align*}
\sup_{B^+_{dy_0}(P_0)} (u-\psi (P_0)) &\leq \sup_{B^+_{2dy_0}(P_1)} (u-\psi (P_0))
\\
&\leq \sup_{B^+_{2dy_0}(P_1)} (u-\psi (P_1)) + | \psi(P_1) - \psi (P_0) |
\\
&\leq Cdy_0\| \psi \|_{C^{1,1}(\bar B^+_{\rho_0}(P_0))},
\end{align*}
for a constant, $C$, depending at most on the coefficients of $L$, and hence
\begin{equation}
\label{eqn-good11}
\sup_{B^+_{dy_0}(P_0)} (u-\psi (P_0)) \leq C dy_0\| \psi \|_{C^{1,1}(\bar V)}.
\end{equation}
We now consider the  function
\begin{equation}
\label{eq:defnw-case2-pfthm}
w:=u-\psi (P_0) \in C_s^{2+\alpha}(B^+_{dy_0}(P_0)\cup\Gamma_0)\cap C(\bar B^+_{dy_0}(P_0)),
\end{equation}
which satisfies the equation
$$
Lw = - L \psi(P_0) = r \psi(P_0) \quad\hbox{on }B^+_{dy_0}(P_0),
$$
since $B_{d y_0}  (P_0) \cap \E(u)  = \emptyset$.  By defining the rescaled function
$$
\bar w(x,y) := \frac {1}{d  y_0}w(x_0 + d y_0 x, y_0 +  d  y_0  y), \quad\forall(x,y) \in D,
$$
on $D:=B_1 \cap \{y > -1/d\} = B_1 \cap \{y_d > 0\}$, we see that
\begin{equation}
\label{eq:Ldbarw}
L_d \bar w = r \psi(P_0) \quad \mbox{on } D,
\end{equation}
with (compare \eqref{eq:LuLvidentity} and \eqref{eqn-Lv})
$$
L_d  \bar w := y_d\left( \bar w_{xx} + 2 \varrho  \sigma \bar w_{xy} + \sigma^2 \bar w_{yy} \right) +
\left(r - q\frac {d  y_0 y_d}2\right)\bar  w_x + \kappa\left(\theta - d y_0 y_d \right)\bar  w_y - r d  y_0 \bar w,
$$
and $y_d :=1/d  + y$. The operator $L_d$ becomes degenerate at $y_d =0$ or  equivalently $y=- 1/d$ which explains why the  domain of consideration in the new variables is the intersection, $D$.

It follows from the bound \eqref{eqn-good11} that
$$
\|\bar w\|_{C(\bar D)} \leq C\| \psi \|_{C^{1,1}(\bar V)},
$$
for a constant, $C$, depending at most on the coefficients of $L$. Denote $D_{1/2}:=B_{1/2} \cap \{y_d > 0\}$. Hence, combining the preceding inequality with the  Schauder estimate \eqref{eqn-schauder} for the solution $\bar w$ to the equation \eqref{eq:Ldbarw} and noting that $y_d = 1/d+y \geq 1/d$ on $\{y\geq 0\}$, yields the bound
\begin{align*}
\frac{1}{d}|  D^2 \bar w (0,0) | + |D \bar w(0,0)| &\leq \|y_dD^2 \bar w\|_{C(\bar D_{1/2})} + \|D \bar w\|_{C(\bar D_{1/2})}
\\
&\leq \|\bar w\|_{C^{2+\alpha}_s(\bar D_{1/2})}
\\
&\leq C\left( \| \bar w \|_{C(\bar D)} + r w(P_0) \right)
\\
&\leq C\| \psi \|_{C^{1,1}(\bar V)},
\end{align*}
recalling that $w(P_0) = dy_0\bar w(0,0)$; here, $C$ is a constant depending at most on the coefficients of $L$. Since  $D^2 \bar w  (0,0) = d   y_0  D^2 w (P_0)$ and $D \bar w  (0,0) =  D w (P_0)$, we obtain
$$
y_0|  D^2w (P_0) | + |Dw (P_0)| \leq C\| \psi \|_{C^{1,1}(\bar V)},
$$
and thus, by \eqref{eq:defnw-case2-pfthm},
$$
y_0|  D^2u (P_0) | + |Du (P_0)| \leq C\| \psi \|_{C^{1,1}(\bar V)},
$$
for a possibly larger constant, $C$, but depending at most on the coefficients of $L$. This implies \eqref{eqn-c11} for this case.
\end{case}
\end{step}

We consider the third situation.

\begin{step}[$\rho_0/4 <   d \leq  1$]
\label{step:pfthm-quarterrhozero_leq_d_leq_one}
This is the simplest situation.  As in Step \ref{step:pfthm-dgeq1}, we consider two cases.

\begin{figure}[htbp]
\centering
\begin{picture}(200,200)(0,0)
\put(0,0){\includegraphics[height=200pt]{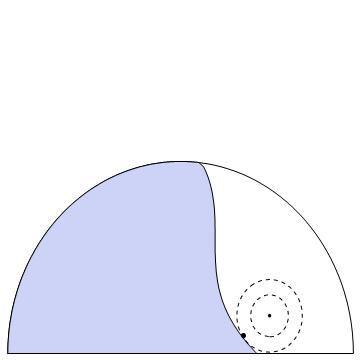}}
\put(45,95){$\scriptstyle V$}
\put(30,20){$\scriptstyle \sE(u)$}
\put(131,53){$\scriptstyle \sC(u)$}
\put(125,10){$\scriptstyle P_1$}
\put(140,20){$\scriptstyle P_0$}
\put(142,37){$\scriptstyle \rho_0/4$}
\put(165,38){$\scriptstyle dy_0$}
\end{picture}
\caption{Regions for Case \ref{case:pfthm-quarterrhozero_leq_d_leq_one_and_dyzero_geq_quarterrhozero} of Step \ref{step:pfthm-quarterrhozero_leq_d_leq_one} (where $\rho_0/4 <   d \leq  1$ and $d y_0 \geq \rho_0/4$) of the proof of Theorem \ref{thm-c11} for estimating the $C^{1,1}_s$ norm of a solution near the free boundary and away from the degenerate boundary.}
\label{fig:c11_estimate_away_degen_intermedtangentball}
\end{figure}

\setcounter{case}{0}
\begin{case}[$\rho_0/4 <   d \leq  1$ and $d y_0 \geq \rho_0/4$]
\label{case:pfthm-quarterrhozero_leq_d_leq_one_and_dyzero_geq_quarterrhozero}
When $d y_0 \geq  \rho_0/4$
(see Fig. \ref{fig:c11_estimate_away_degen_intermedtangentball}), we have $Lu=0$ in $B_{\rho_0/4}^+(P_0)$, the estimate \eqref{eq:c11u-schauder} for $u$ holds, and \eqref{eqn-c11} follows in this case.
\end{case}

\begin{case}[$\rho_0/4 <   d \leq  1$ and $d y_0 < \rho_0/4$]
\label{case:pfthm-quarterrhozero_leq_d_leq_one_and_smallyzero}
We now assume that $d y_0 < \rho_0/4$ (see Fig. \ref{fig:c11_estimate_away_degen_intermedtangentball_smalldyzero}). We consider
\begin{equation}
\label{eq:defnw-case3-pfthm}
w:=u-\psi(P_0) \in C^{2,\alpha}(B_{dy_0}(P_0))\cap C(\bar B_{dy_0}(P_0)),
\end{equation}
and the rescaled function,
$$
\bar w(x,y) := \frac {1}{y_0}w(x_0 +  y_0 x, y_0 +  y_0  y), \quad (x,y) \in B_d,
$$
which   satisfies (compare \eqref{eq:LuLvidentity})
$$
\bar L \bar w = r \psi(P_0)   \quad \mbox{on } B_d,
$$
with (compare \eqref{eqn-Lv})
\begin{align*}
\bar L  \bar w &:= \frac{1+y}2 \left( \bar w_{xx} + 2 \varrho  \sigma \bar w_{xy} + \sigma^2 \bar w_{yy} \right) +
\left(r - q-\frac {y_0(1+y)}2\right)\bar w_x
\\
&\qquad + \kappa\left(\theta - y_0 (1+y)\right)\bar w_y - ry_0 \bar w.
\end{align*}
The operator $\bar L$ is strictly elliptic on $B_d$ with ellipticity constant bounded below by a positive constant depending at most on the coefficients of $L$. In addition, since $dy_0 < \rho_0/4$, the bound \eqref{eqn-good11} applies (irrespective of whether  $d\leq 1$ or $d>1$)
to give
$$
|w| \leq Cdy_0\| \psi \|_{C^{1,1}(\bar V)} \quad\hbox{on } B_{dy_0}(P_0),
$$
and thus
$$
|\bar w| \leq Cd\| \psi \|_{C^{1,1}(\bar V)} \quad\hbox{on }B_{d},
$$
for a constant, $C$, depending at most on the coefficients of $L$.
%
\begin{figure}[htbp]
\centering
\begin{picture}(200,200)(0,0)
\put(0,0){\includegraphics[height=200pt]{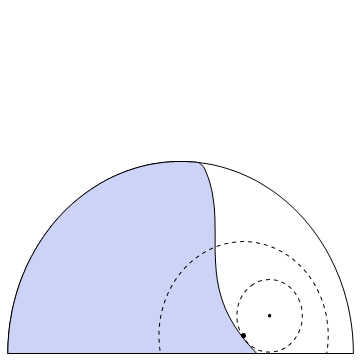}}
\put(45,95){$\scriptstyle V$}
\put(30,20){$\scriptstyle \sE(u)$}
\put(125,75){$\scriptstyle \sC(u)$}
\put(125,10){$\scriptstyle P_1$}
\put(140,20){$\scriptstyle P_0$}
\put(138,45){$\scriptstyle dy_0$}
\put(146,66){$\scriptstyle \rho_0$}
\end{picture}
\caption{Regions for Case \ref{case:pfthm-quarterrhozero_leq_d_leq_one_and_smallyzero} of Step \ref{step:pfthm-quarterrhozero_leq_d_leq_one} (where $\rho_0/4 <   d \leq  1$ and $d y_0 < \rho_0/4$) of the proof of Theorem \ref{thm-c11} for estimating the $C^{1,1}_s$ norm of a solution near the free boundary and away from the degenerate boundary. (The radii $dy_0$ and $\rho_0$ are not drawn to scale, since $dy_0<\rho/4$ for this case.)}
\label{fig:c11_estimate_away_degen_intermedtangentball_smalldyzero}
\end{figure}
Combining the preceding estimate with the classical Schauder interior estimate \cite[Corollary 6.3]{GilbargTrudinger} gives
\begin{align*}
|D^2 \bar w(0)| + |D\bar w (0)|  &\leq \|\bar  w \|_{C^{2,\alpha}(\bar B_{d/2})}
\\
&\leq C\left( \|\bar  w \|_{C(\bar B_{d}) }+ |\psi(P_0)|\right)
\\
&\leq C\| \psi \|_{C^{1,1}(\bar V)},
\end{align*}
again for a constant, $C$, depending at most on the coefficients of $L$ (recall that $\rho_0/4\leq d\leq 1$ in this case and that $\rho_0$ depends at most on the coefficients of $L$). Hence,
$$
y_0|D^2w(P_0)| + |Dw (P_0)|  \leq C\| \psi \|_{C^{1,1}(\bar V)},
$$
since  $D^2 \bar w  (0,0) = y_0  D^2 w (P_0)$ and $D \bar w  (0,0) =  D w (P_0)$. Thus, by \eqref{eq:defnw-case3-pfthm},
$$
y_0|D^2u(P_0)| + |Du (P_0)|  \leq C\| \psi \|_{C^{1,1}(\bar V)},
$$
for a possibly larger constant, $C$, and \eqref{eqn-c11} follows in this case too.
\end{case}
\end{step}
This completes the proof of Theorem \ref{thm-c11}.
\end{proof}

%
%

\bibliography{mfpde}
\bibliographystyle{amsplain}

\end{document}